\documentclass[onefignum,onetabnum,onealgnum]{siamonline250211}

\usepackage{amsfonts,amssymb}
\usepackage{mathtools}
\usepackage{algorithmic}
\usepackage{subcaption}
\usepackage{enumitem}
\usepackage{cite}
\setlist[enumerate]{leftmargin=.5in}
\setlist[itemize]{leftmargin=.5in}

\newsiamthm{assumption}{Assumption}
\newsiamremark{remark}{Remark}
\crefname{assumption}{Assumption}{Assumptions}
\crefname{remark}{Remark}{Remarks}
\AddToHook{env/assumption/begin}{\crefalias{theorem}{assumption}}
\AddToHook{env/remark/begin}{\crefalias{theorem}{remark}}

\newcommand{\bbE}{\mathbb{E}}
\newcommand{\bbP}{\mathbb{P}}
\newcommand{\bbR}{\mathbb{R}}
\newcommand{\bbS}{\mathbb{S}}
\newcommand{\cA}{\mathcal{A}}
\newcommand{\cD}{\mathcal{D}}
\newcommand{\cF}{\mathcal{F}}
\newcommand{\cL}{\mathcal{L}}
\newcommand{\cM}{\mathcal{M}}
\newcommand{\cO}{\mathcal{O}}
\newcommand{\cP}{\mathcal{P}}
\newcommand{\cR}{\mathcal{R}}
\newcommand{\cS}{\mathcal{S}}
\newcommand{\cT}{\mathcal{T}}
\newcommand{\bs}[1]{\boldsymbol{#1}}
\makeatletter
\let\liu@accentv\v
\DeclareRobustCommand{\v}[1]{\ifmmode\mathbf{#1}\else\liu@accentv{#1}\fi}
\let\liu@accentt\t
\DeclareRobustCommand{\t}{\ifmmode\tilde{\v t}\else\expandafter\liu@accentt\fi}
\makeatother
\newcommand{\n}{\tilde{\v n}}
\newcommand{\sff}{\mathrm{I\!I}}
\newcommand{\defeq}{\mathrel{\mathop:}=}
\newcommand{\disteq}{\buildrel d \over =}
\newcommand{\dd}[1]{\mathrm{d}#1}
\newcommand{\inpd}[2]{\langle #1\,,\, #2\rangle}
\newcommand{\abs}[1]{\lvert #1 \rvert}
\newcommand{\absmid}[1]{\big\lvert #1 \big\rvert}
\newcommand{\absbig}[1]{\Big\lvert #1 \Big\rvert}
\newcommand{\norm}[1]{\lVert #1 \rVert}
\newcommand{\normbig}[1]{\Big\lVert #1 \Big\rVert}
\newcommand{\evalmid}[2]{\Big\vert^{#1}_{#2}}
\newcommand{\Cov}{\mathrm{Cov}}
\DeclareMathOperator*{\argmin}{arg\,min}
\DeclareMathOperator{\dist}{dist}
\DeclareMathOperator{\Proj}{Proj}
\DeclareMathOperator{\tr}{tr}
\DeclareMathOperator{\vect}{vec}
\DeclareMathOperator{\diag}{diag}
\DeclareMathOperator{\Hess}{Hess}
\DeclareMathOperator{\median}{median}

\headers{Extrinsic Curvature Estimation by Closest Point Projection of Rays}{L. Liu and R. Tsai}

\title{Extrinsic Curvature Estimation by Closest Point Projection of Rays}

\author{Liangchen Liu\thanks{Department of Mathematics, the University of Texas at Austin, Austin, TX
  (\email{lcliu@utexas.edu}).}
\and Richard Tsai\footnotemark[1] \thanks{Oden Institute, the University of Texas at Austin, Austin, TX
  (\email{ytsai@math.utexas.edu}).}}

\ifpdf
\hypersetup{
  pdftitle={Extrinsic Curvature Estimation by Closest Point Projection of Rays},
  pdfauthor={L. Liu and R. Tsai}
}
\fi

\begin{document}
\maketitle

\begin{abstract}
We propose a parametrization-free framework for estimating the second fundamental form and the principal curvatures of a submanifold of $\mathbb{R}^d$. The method uses the closest point projections of a few short random rays and recovers the curvature from the curves that these projections induce on the submanifold, without coordinate charts or local fitting. We analyze the geometry of the induced curves and characterize when they are geodesics. For hypersurfaces, these results give two algorithms with error bounds: a least squares reconstruction of the second fundamental form from the curvatures of the induced curves, and a direct estimate of all principal curvatures from the eigenvalues of a $d\times d$ matrix built from $d$ rays, which needs no normal information and costs $\mathcal{O}(d^3)$ besides the projections. In higher codimension, we can also recover the second fundamental form one normal direction at a time. On point clouds, where the projection becomes a nearest neighbor query, we bound the additional error. Experiments confirm the predicted rates for hypersurfaces in up to $128$ dimensions, for surfaces of codimension up to $98$, and on point clouds, demonstrating the effectiveness of the proposed framework in practice.

\end{abstract}



\section{Introduction}\label{sec:curv:intro}

Understanding the geometry of smooth manifolds embedded in Euclidean space is a fundamental problem in differential geometry and its applications. Given a smooth submanifold $\cM\subset\bbR^d$, most simply a hypersurface, a central question is how to compute its extrinsic curvature---the second fundamental form and principal curvatures---from the embedding alone, without constructing explicit local coordinate charts or parametrizations.

Point clouds arising in applications, including imaging, genomics, and physical simulation, are often observed to concentrate near low-dimensional geometric structures~\cite{carlsson2009topology, dicarlo2012does, haghverdi2015diffusion, moon2019visualizing, pope2021intrinsic}, which is known as the manifold hypothesis~\cite{fefferman2016testing}. Exploiting geometric quantities of these structures, such as reach, intrinsic dimension, and Laplace--Beltrami operators, has been shown to improve learning and inference~\cite{belkin2004semi, allard2012multi, aamari2019estimating}, which highlights the importance of understanding the geometry of the underlying data distribution.

A natural tool for studying embedded manifolds is the \emph{closest point projection}, the map that sends each point in the ambient space to the nearest point on $\cM$. The \emph{reach} $\iota$ of $\cM$ is the largest number such that every point at distance less than $\iota$ from $\cM$ has a unique closest point on $\cM$~\cite{federer1959curvature}. It is well known that, within the reach of a $C^2$ hypersurface, this projection is well defined and continuously differentiable, and the tubular neighborhood of that width is diffeomorphic to $\cM\times(-\iota,\iota)$~\cite{federer1959curvature, leobacher2021existence}. The closest point projection and the associated tubular neighborhood structure have been exploited extensively in the numerical solution of partial differential equations on surfaces~\cite{ruuth2008simple, marz2012calculus, chen2015closest}, where they provide a principled way of extending surface-related quantities to the ambient space.

A fundamental geometric quantity, \emph{curvature} describes how a manifold bends and shapes the local behavior of distances and neighborhoods. In applied geometry, curvatures serve as structural descriptors in surface-related tasks such as fitting~\cite{wang2006fitting}, segmentation~\cite{rabbani2006segmentation}, and geometric inference~\cite{pauly2002efficient}. Moreover, curvature information can provide insight into both the robustness and interpretability of machine learning models. For instance, curvatures arise in the analysis of sample complexity, hardness of learning tasks, and generalization error~\cite{narayanan2010sample, shaham2018provable, kiani2024hardness, wang2024manifold}. They have also been linked to structural properties of the optimization landscape in the training of deep neural networks~\cite{chaudhari2019entropy, pope2021intrinsic, pouplin2023curvature}. Even in linear models, it has been observed in~\cite{he2023side, liu2023linear} that low intrinsic dimensionality combined with curvatures can influence regression behavior in nontrivial ways.

Among various notions of curvatures, \emph{extrinsic curvature} is the natural consideration when the manifold is embedded in Euclidean space and observed through ambient coordinates. A substantial body of work in computer vision and geometry processing has studied the estimation of extrinsic curvatures~\cite{berkmann1994computation, meek2000surface, rusinkiewicz2004estimating, lange2005anisotropic, merigot2010voronoi}, primarily in low dimensions and under structured representations such as triangular meshes or reconstructed surfaces. Extending such approaches to higher dimensions is challenging, as techniques effective in $\bbR^3$ often become computationally prohibitive or ill-posed~\cite{klee1980complexity}.
More broadly, many existing approaches proceed by introducing auxiliary local geometric structure, such as local parametric approximations~\cite{levin1998approximation, cazals2005estimating, lai2013local} or tangent space estimates obtained via local principal component analysis~\cite{hoppe1992surface, cheng2013local, aamari2019nonasymptotic, cao2021efficient}. While effective in moderate dimensions, these methods depend on the stable construction of local coordinate systems and accurate tangent information. Both requirements become increasingly fragile and computationally expensive as the ambient dimension grows.

Our approach avoids such local structure: it only uses the closest point projections of random rays, making it suited to high-dimensional ambient spaces. In~\cite{liu2024nearest}, rays were used to sample point clouds: a ray probes a point cloud through the closest point projection onto the point cloud, a single ray only records the Voronoi cells it visits, and statistical information about the point cloud comes from averages over many random rays. In this paper, we use rays in a different way. We assume that the points are sampled densely from a smooth manifold $\cM$, and we start the ray at a point close to $\cM$. As we move along the ray, its closest point on $\cM$ moves along $\cM$ and traces a curve, the curve induced by the ray. Near the starting point, this curve is determined by the curvature of $\cM$ at the closest point. Since the information is local, a few rays from one starting point are enough, and the directions of the rays are random only so that no normal information about $\cM$ is needed.

\subsection{Methodologies and setup}\label{sec:curv:setup}
Throughout the paper, except in \cref{sec:curv:codim}, let $\cM \subset \bbR^d$ be a compact, connected $C^2$ hypersurface with positive reach $\iota>0$. Our objective is to efficiently compute the extrinsic curvature of $\cM$.

Let $\v x_0 \in \bbR^d$ be a point sampled uniformly at random in a bounded region near $\cM$, and let $\v r \in \bbR^d$ be a direction. The set $\{\ell(h) \defeq \v x_0 + h\v r : h\ge 0\}$ defines a ray in $\bbR^d$ emitting from $\v x_0$. The closest point projection onto $\cM$ is defined by
\begin{equation}\label{eqn:curv:cp_def}
\cP_{\cM}\v x\defeq\argmin_{\v y\in\cM} \norm{\v x - \v y}.
\end{equation}
For every $\v x$ within the reach of $\cM$, the minimizer is unique; we write $\cP\defeq\cP_\cM$ when no ambiguity arises, and $\cP_\Gamma$ for the projection onto a point cloud $\Gamma$ (\cref{sec:curv:point_cloud}). Although the definition is the same as for point clouds, the projection behaves differently: $\cP_\Gamma$ is constant on each Voronoi cell, while $\cP$ is differentiable within the reach (\cref{sec:curv:cp_property}), and this is what lets a ray sense the curvature of $\cM$. Our method proceeds as follows:
\begin{enumerate}
  \item Let $\v y_0\defeq\cP\v x_0\in\cM$ be the closest point to $\v x_0$ on $\cM$.
  \item For small $h>0$, project the ray $\ell(h) = \v x_0 + h \v r$ passing through $\v x_0$ onto $\cM$.
  \item The projection $\{\cP(\ell(h)) = \cP(\v x_0+h\v r)\}$ induces a curve on $\cM$ as demonstrated in~\cref{fig:curv:local_frame}.
  \item Analyzing projections along a set of linearly independent $\v r$, we recover the curvatures of $\cM$ at $\v y_0$.
\end{enumerate}

\begin{figure}[htbp]
  \centering
  \begin{subfigure}{.5\textwidth}
    \centering
    \includegraphics[width=.75\linewidth]{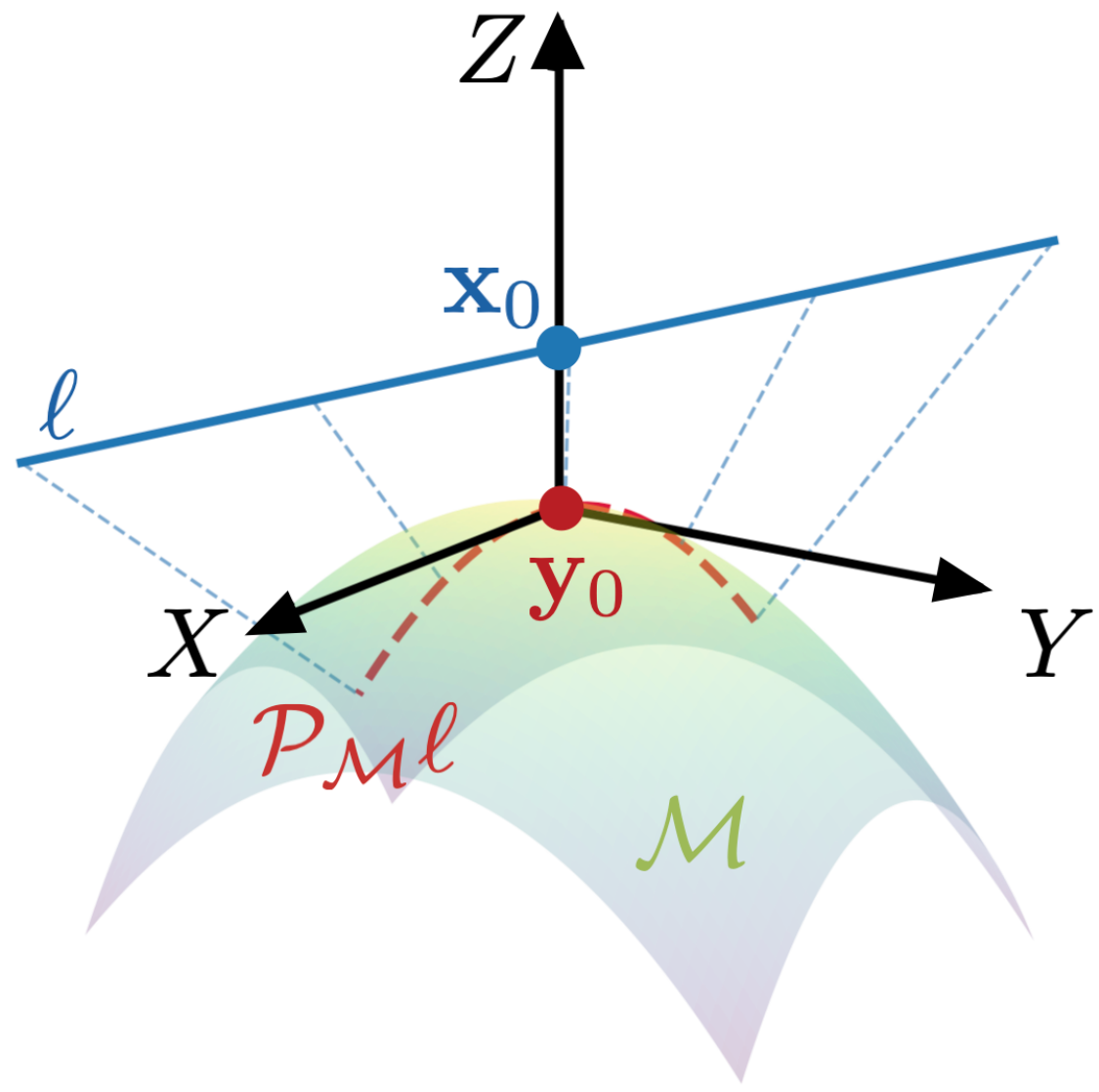}
    \caption{Viewed under a local frame.}
    \label{fig:curv:local_frame}
  \end{subfigure}%
  \begin{subfigure}{.5\textwidth}
    \centering
    \includegraphics[width=.9\linewidth]{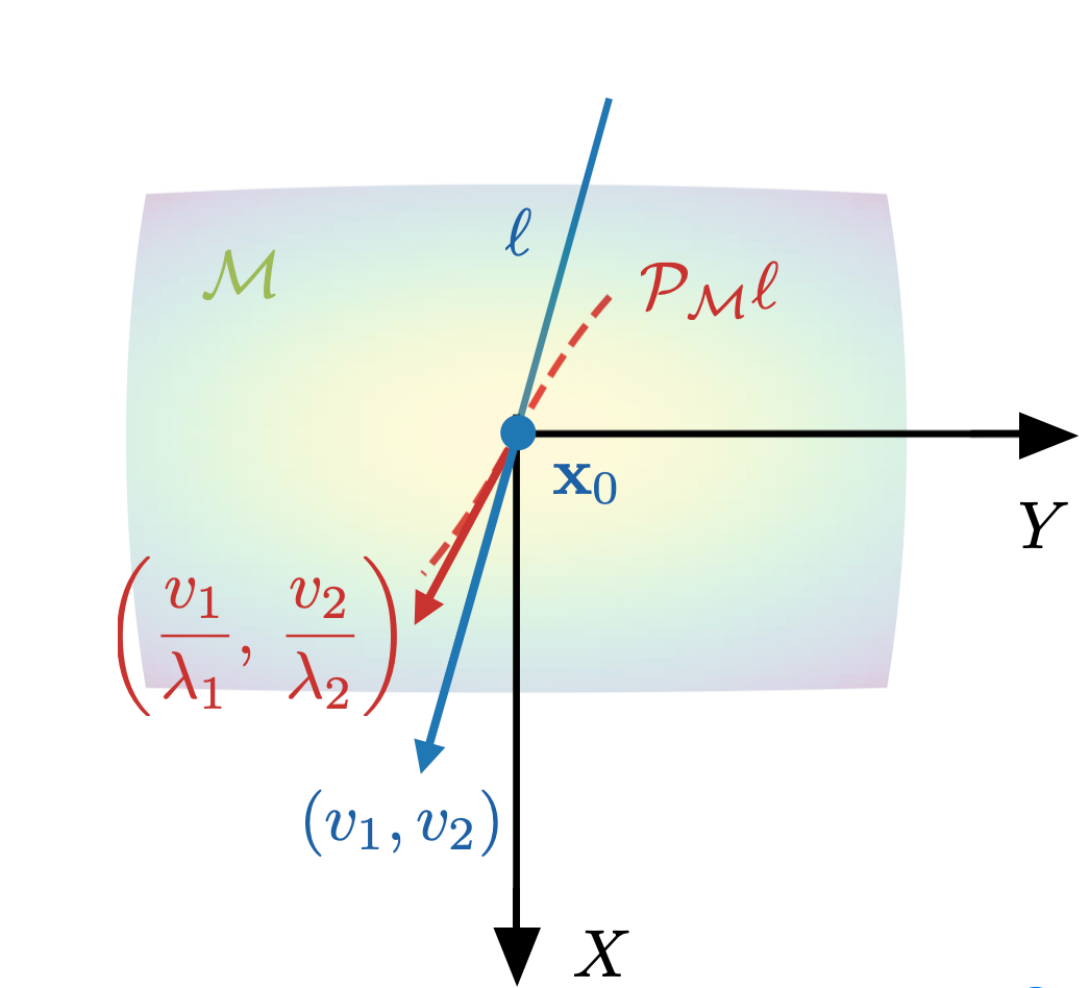}
    \caption{Top view of the projections (view of the tangent space).}
    \label{fig:curv:local_deviation}
  \end{subfigure}
  \caption{Closest point projection of $\ell$ onto $\cM$. The projection generates an induced curve $\cP\ell$ on $\cM$.}
  \label{fig:curv:cls_pt_proj_demo}
\end{figure}
The assumption for our setup is the following.
\begin{assumption}\label{asm:curv:standing}
$\cM\subset\bbR^d$ is a compact, connected $C^2$ hypersurface with reach $\iota>0$, and every ambient point at which $\cP$ is evaluated lies within distance $\iota$ of $\cM$.
\end{assumption}
We do not assume that $\cM$ is convex, although the surface in~\cref{fig:curv:cls_pt_proj_demo} is. Only the later \cref{sec:curv:convergence}, where $\v x_0$ moves far away from $\cM$, and the results that use it need such an assumption.

\subsection{Main contributions}
Our main contributions are summarized as follows:
\begin{enumerate}
\item \textbf{Geometric analysis of the closest point projection-induced curves.}
We show that the direction of the induced curve at $\v y_0$ is given exactly by the shape operator of $\cM$ and the offset distance, and that the normal part of its acceleration is the normal curvature of $\cM$. We establish that the induced curve is geodesic at $\v y_0\in\cM$ whenever the probing direction $\v r$ lies in the tangent space $\cT_{\v y_0}\cM$, or its projection onto the tangent space is a principal direction. In general, the induced curve only converges to a geodesic as the offset distance $\norm{\v x_0-\v y_0}\to\infty$, provided that $\cM$ is locally convex near $\v y_0$ and $\v x_0$ lies along the outward normal. We derive explicit expressions for the direction and curvature of this limiting geodesic, and quantify the deviation of the induced curve from being geodesic when these conditions are not satisfied. These statements hold to second order at $\v y_0$: the third-order terms of $\cM$ at $\v y_0$ add a geodesic curvature, which is small for small offsets but does not vanish as the offset grows.

\item \textbf{Parametrization-free curvature estimation.}
Building on the insights from the analysis of the induced curves, we develop two randomized approaches for estimating the second fundamental form and the principal curvatures directly from closest-point projections of one-dimensional rays, thereby avoiding the explicit construction of local coordinate charts.

\item \textbf{Error analysis and numerical validation.}
We establish explicit error bounds and convergence rates in terms of algorithmic hyperparameters, without explicit tangent-space recovery, with errors decaying at a controlled rate with high probability. We also provide numerical experiments to validate the theoretical rates and demonstrate the stability and scalability of the method on hypersurfaces up to $d = 128$. On point clouds, we bound the additional discretization error that comes from replacing the closest point projection by a nearest neighbor query.

\item \textbf{Submanifolds of higher codimension.} A probing point recovers the shape operator only in the normal direction along its offset. Using the structure of the closest point projection near a submanifold of any codimension~\cite{marz2012calculus,leobacher2021existence}, we place probing points along a basis of the normal space at one closest point and recover the second fundamental form one normal direction at a time. The idea follows the construction of closest point functions one codimension at a time by M\"arz and Macdonald~\cite{marz2012calculus}. Their closest point functions can replace the closest point projection, at the price of a curvature error of first order in the offset distance.

\item \textbf{A projection-based geometric sensing framework.}
Our approach introduces a new perspective for recovering geometric information through closest-point projections of randomized ambient probes. While the closest point projection is a classical tool in differential geometry and has been widely used in numerical methods for surface PDEs~\cite{ruuth2008simple, marz2012calculus}, its direct application to curvature recovery has not, to our knowledge, been previously explored. The present work lays the theoretical foundation for such an approach and may serve as a starting point for broader closest-point-projection-based methods for geometric analysis of embedded manifolds.
\end{enumerate}

We recall the extrinsic geometry that we use (\cref{sec:curv:prelim}), and we collect some useful properties of the closest point projection (\cref{sec:curv:cp_property}). We then study the curves induced by rays on a hypersurface (\cref{sec:curv:induced_curve}) and turn the results into two curvature algorithms with error bounds (\cref{sec:curv:curv_compute}). We extend them to submanifolds of higher codimension (\cref{sec:curv:codim}) and apply the second algorithm to point clouds (\cref{sec:curv:point_cloud}). We summarize the results in \cref{sec:curv:conclusion}. The main point of this paper is that the closest point projections of a few rays give the extrinsic curvature of a manifold in any dimension and codimension, without local coordinate charts or local fitting.

\section{Geometric preliminaries}\label{sec:curv:prelim}

This section collects the classical notions of extrinsic geometry that we use; proofs can be found in~\cite{o2006elementary,singer2015lecture}. Throughout, $\cM\subset\bbR^d$ is an embedded submanifold of class at least $C^2$, $\cT_{\v y}\cM$ is its tangent space at $\v y\in\cM$, and $\cT^\perp_{\v y}\cM\defeq(\cT_{\v y}\cM)^\perp$ is its normal space. Vectors are in bold, $\Proj_L$ is the orthogonal projection onto a linear subspace $L$, $\cD f$ is the derivative of a map $f$, and a hat marks an estimate, for example, $\hat\kappa_i$ is an estimate of $\kappa_i$.

\subsection{Hypersurfaces}\label{sec:curv:prelim_hyper}

Let $\cM$ be a hypersurface, i.e., $\dim\cM=d-1$. Near each of its points, $\cM$ is the zero level set of a $C^2$ function $F$ with $\nabla F\neq\v 0$ on $\cM$, and $\v n=\nabla F/\norm{\nabla F}$ is a unit normal vector field; the normal space at $\v y$ is spanned by $\v n(\v y)$.

\begin{definition}[Weingarten map and second fundamental form]\label{def:curv:2nd_fund}
Let $\v n$ be a unit normal vector field of $\cM$ around $\v y\in\cM$. The linear operator
\begin{equation}\label{eqn:curv:shape_operator}
S_{\v y}:\cT_{\v y}\cM\to\cT_{\v y}\cM,\qquad S_{\v y}\v v\defeq-\cD_{\v v}\v n(\v y),
\end{equation}
where $\cD_{\v v}$ is the directional derivative along $\v v\in\cT_{\v y}\cM$, is called the \emph{Weingarten map} or \emph{shape operator} of $\cM$ at $\v y$. The bilinear form
\[
\sff_{\v y}(\v v_1,\v v_2)\defeq\inpd{S_{\v y}\v v_1}{\v v_2},\qquad\v v_1,\v v_2\in\cT_{\v y}\cM,
\]
is the \emph{second fundamental form} of $\cM$ at $\v y$ (with respect to $\v n$).
\end{definition}

Differentiating $\inpd{\v n}{\v n}\equiv1$ along $\v v$ gives $\inpd{S_{\v y}\v v}{\v n(\v y)}=-\frac12\cD_{\v v}\inpd{\v n}{\v n}=0$, so $S_{\v y}$ indeed maps $\cT_{\v y}\cM$ into itself. Moreover, $S_{\v y}$ is self-adjoint, i.e., $\sff_{\v y}$ is symmetric. Its eigenvalues $\kappa_1,\dots,\kappa_{d-1}$, counted with multiplicity, are the \emph{principal curvatures} of $\cM$ at $\v y$, and an orthonormal basis of eigenvectors gives the \emph{principal directions}. For a unit vector $\v v\in\cT_{\v y}\cM$, $\sff_{\v y}(\v v,\v v)$ is the \emph{normal curvature} of $\cM$ in the direction $\v v$: if $\alpha$ is a unit-speed curve on $\cM$ with $\alpha(0)=\v y$ and $\alpha'(0)=\v v$, differentiating $\inpd{\alpha'}{\v n\circ\alpha}=0$ gives $\inpd{\alpha''(0)}{\v n(\v y)}=\sff_{\v y}(\v v,\v v)$. Reversing the orientation of $\v n$ changes the signs of $S_{\v y}$, $\sff_{\v y}$ and the $\kappa_i$.

In computations it is convenient to describe $\cM$ near a point as a graph over its tangent space.

\begin{proposition}[Local graph representation]\label{thm:curv:monge}
Let $\v y\in\cM$, and choose coordinates $(\v s,z)\in\bbR^{d-1}\times\bbR$ such that $\v y=\v 0$, $\cT_{\v 0}\cM=\bbR^{d-1}\times\{0\}$ and $\v n(\v 0)=\v e_d$. Then, in a neighborhood $U$ of $\v 0$,
\[
\cM\cap U=\big\{\big(\v s,f(\v s)\big)\big\}
\]
for a $C^2$ function $f$ defined near $\v 0\in\bbR^{d-1}$ with $f(\v 0)=0$ and $\nabla f(\v 0)=\v 0$, and:
\begin{enumerate}[label=(\roman*)]
  \item the shape operator at $\v 0$ is $S_{\v 0}=\Hess(f)(\v 0)$, acting on $\cT_{\v 0}\cM=\bbR^{d-1}$; in particular, the principal curvatures of $\cM$ at $\v 0$ are the eigenvalues of $\Hess(f)(\v 0)$;
  \item if the coordinates $s_1,\dots,s_{d-1}$ are taken along the principal directions, then
    \begin{equation}\label{eqn:curv:monge}
      z=f(\v s)=\frac12\sum_{i=1}^{d-1}\kappa_is_i^2+o(\norm{\v s}^2),
    \end{equation}
    and the remainder is $\cO(\norm{\v s}^3)$ if $\cM$ is of class $C^3$.
\end{enumerate}
\end{proposition}
\begin{proof}
By the implicit function theorem, $\cM$ is locally the graph of a $C^2$ function $f$ over its tangent plane, and $f(\v 0)=0$, $\nabla f(\v 0)=\v 0$ because the graph passes through $\v 0$ and is tangent to $\bbR^{d-1}\times\{0\}$ there. With $F(\v s,z)=z-f(\v s)$ we have $\nabla F=(-\nabla f(\v s),1)$, so $\v n=\nabla F/\norm{\nabla F}$ satisfies $\v n(\v 0)=\v e_d$. For $\v v=(\v v',0)\in\cT_{\v 0}\cM$, the derivative of $\norm{\nabla F}^2=1+\norm{\nabla f}^2$ along $\v v$ is $2\inpd{\nabla f(\v 0)}{\Hess(f)(\v 0)\v v'}=0$ at $\v 0$. Hence
\[
S_{\v 0}\v v=-\cD_{\v v}\v n(\v 0)=-\cD_{\v v}\nabla F(\v 0)=\big(\Hess(f)(\v 0)\v v',\,0\big),
\]
which proves (i). Taylor's theorem gives $f(\v s)=\frac12\v s^T\Hess(f)(\v 0)\v s+o(\norm{\v s}^2)$, with remainder $\cO(\norm{\v s}^3)$ when $f$ is $C^3$. Since $\Hess(f)(\v 0)$ is symmetric, a rotation of the coordinates $\v s$ diagonalizes it, with the principal curvatures on the diagonal by (i).
\end{proof}

No sign condition is involved here: $\Hess(f)(\v 0)$ is symmetric but in general indefinite, and the principal curvatures may have different signs, as at a saddle point. If $\cM$ is locally convex near $\v y$, then $f$ is convex or concave near $\v 0$ and $\Hess(f)(\v 0)$ is semidefinite. In this work, we write the quadratic part of $f$ as $\v s^TA\v s$; thus $A=\frac12\Hess(f)(\v 0)$, and the principal curvatures are the eigenvalues of $2A$.

\subsection{Submanifolds of higher codimension}\label{sec:curv:prelim_codim}

Let $\cM$ have dimension $d_{\mathrm{in}}$ and codimension $d_{\mathrm{co}}=d-d_{\mathrm{in}}\ge1$. Every normal vector $\bs\nu\in\cT^\perp_{\v y}\cM$ defines a shape operator in the same way as for a hypersurface:
\begin{equation}\label{eqn:curv:shape_operator_codim}
S_{\bs\nu}:\cT_{\v y}\cM\to\cT_{\v y}\cM,\qquad S_{\bs\nu}\v v\defeq-\Proj_{\cT_{\v y}\cM}\cD_{\v v}\tilde{\bs\nu}(\v y),
\end{equation}
where $\tilde{\bs\nu}$ is any normal vector field near $\v y$ with $\tilde{\bs\nu}(\v y)=\bs\nu$; the result does not depend on the choice of $\tilde{\bs\nu}$, $S_{\bs\nu}$ is self-adjoint, and it is linear in $\bs\nu$~\cite{o2006elementary,leobacher2021existence}. The eigenvalues of $S_{\bs\nu}$ for a unit vector $\bs\nu$ are the principal curvatures of $\cM$ in the normal direction $\bs\nu$. The \emph{second fundamental form} is now vector valued: $\sff_{\v y}:\cT_{\v y}\cM\times\cT_{\v y}\cM\to\cT^\perp_{\v y}\cM$ is defined by
\begin{equation}\label{eqn:curv:sff_vector}
\inpd{\sff_{\v y}(\v v,\v w)}{\bs\nu}=\inpd{S_{\bs\nu}\v v}{\v w}\qquad\text{for all }\bs\nu\in\cT^\perp_{\v y}\cM,
\end{equation}
so that $\sff_{\v y}(\v v,\v w)=\sum_{j=1}^{d_{\mathrm{co}}}\inpd{S_{\bs\nu_j}\v v}{\v w}\bs\nu_j$ for an orthonormal basis $\bs\nu_1,\dots,\bs\nu_{d_{\mathrm{co}}}$ of $\cT^\perp_{\v y}\cM$, and $\cM$ is described to second order by the $d_{\mathrm{co}}$ symmetric matrices $S_{\bs\nu_j}$. For a hypersurface with unit normal $\v n$, $S_{\v n}=S_{\v y}$ and $\sff_{\v y}(\v v,\v w)=\inpd{S_{\v y}\v v}{\v w}\v n$, as in \cref{def:curv:2nd_fund}.

The local graph representation also extends. With coordinates $(\v s,\v z)\in\bbR^{d_{\mathrm{in}}}\times\bbR^{d_{\mathrm{co}}}$ such that $\v y=\v 0$ and $\cT_{\v 0}\cM=\bbR^{d_{\mathrm{in}}}\times\{\v 0\}$, $\cM$ is locally the graph $\v z=f(\v s)=(f_1(\v s),\dots,f_{d_{\mathrm{co}}}(\v s))$ of a $C^2$ map with $f(\v 0)=\v 0$ and $\cD f(\v 0)=0$. The proof of \cref{thm:curv:monge}, applied to each $F_j(\v s,\v z)=z_j-f_j(\v s)$, shows that the shape operator in the $j$-th normal coordinate direction is $S_{\v e_{d_{\mathrm{in}}+j}}=\Hess(f_j)(\v 0)$. Therefore, $f(\v s)=\frac12\sff_{\v 0}(\v s,\v s)+o(\norm{\v s}^2)$. The $d_{\mathrm{co}}$ matrices $\Hess(f_j)(\v 0)$ need not commute, so in general no single rotation of $\v s$ diagonalizes all of them.

Let $\alpha$ be a $C^2$ curve on $\cM$ with $\alpha(0)=\v y$, in any parametrization. Differentiating $\inpd{\alpha'}{\tilde{\bs\nu}\circ\alpha}=0$ at $0$ gives $\inpd{\alpha''(0)}{\bs\nu}=\inpd{S_{\bs\nu}\alpha'(0)}{\alpha'(0)}$ for every normal vector $\bs\nu$, that is,
\begin{equation}\label{eqn:curv:meusnier}
\Proj_{\cT^\perp_{\v y}\cM}\alpha''(0)=\sff_{\v y}\big(\alpha'(0),\alpha'(0)\big).
\end{equation}
So all curves on $\cM$ through $\v y$ with the same velocity have the same normal part of the acceleration at $\v y$, and they differ only in the tangential part (Meusnier's theorem). For a unit-speed curve, $\sff_{\v y}(\alpha'(0),\alpha'(0))$ is the \emph{normal curvature vector} of $\cM$ in the direction $\alpha'(0)$, and the tangential part $\Proj_{\cT_{\v y}\cM}\alpha''(0)$, which is orthogonal to $\alpha'(0)$, is the \emph{geodesic curvature vector} of $\alpha$; the curve is a \emph{geodesic} if its geodesic curvature vanishes along it. For a hypersurface, the normal part is $\kappa_n\v n$ with the normal curvature $\kappa_n=\sff_{\v y}(\alpha'(0),\alpha'(0))$ of \cref{sec:curv:prelim_hyper}.

\section{Properties of the closest point projection}\label{sec:curv:cp_property}

In this section, we collect several well-known properties of the closest point projection $\cP$ that will be used throughout the development. By \cref{asm:curv:standing}, all ambient points lie within the reach of $\cM$, so $\cP$ is uniquely defined and continuously differentiable~\cite{federer1959curvature, marz2012calculus}. We include the short proofs.

\begin{proposition}\label{thm:curv:cp_normal}
For any $\v x$ within the reach of $\cM$, $\v x - \cP\v x$ lies along the normal of $\cM$ at $\cP\v x$.
\end{proposition}
\begin{proof}
Let $F$ denote a local level-set function for $\cM$ in a neighborhood of $\cP\v x$. Then $\nabla F$ is a local choice of the normal vector field of $\cM$. On the other hand, the closest point projection $\cP$ is characterized as the solution to the constrained minimization problem
\[
\min_{\v y}\norm{\v x-\v y}^2, \quad \text{subject to } F(\v y)=0.
\]
The associated Lagrangian is $\cL(\v y) = \norm{\v x-\v y}^2 + \lambda F(\v y)$. A necessary condition for a minimizer $\v y^*$ is therefore
\[
\nabla_{\v y}\cL(\v y^*) = 2(\v y^* - \v x) + \lambda \nabla F(\v y^*) = 0  \text{ for some } \lambda \implies \v x - \cP \v x = c\nabla F(\cP \v x),
\]
for some scalar $c$, which shows that the displacement is normal to $\cM$ at $\cP\v x$.
\end{proof}

The distance to $\cM$ is $\dist(\v x,\cM)\defeq\min_{\v y\in\cM}\norm{\v x-\v y}$; within the reach of $\cM$, it equals $\norm{\v x - \cP\v x}$. It satisfies the following.
\begin{proposition}[The distance to the surface is $1$-Lipschitz]\label{thm:curv:cp_contract}
For any $\v x_1,\v x_2 \in \bbR^d$,
\begin{equation}\label{eqn:curv:dist_lipschitz}
\big|\, \dist(\v x_1,\cM) - \dist(\v x_2,\cM) \,\big| \;\le\; \norm{\v x_1-\v x_2}.
\end{equation}
\end{proposition}
\begin{proof}
By the definition of $\dist$, $\dist(\v x, \cM)\le \norm{\v x - \v y}$ for any $\v y\in\cM$. Therefore
\[
\norm{\v x_1 - \v y} \le \norm{\v x_1 - \v x_2} + \norm{\v x_2-\v y}
\implies \norm{\v x_1 - \v x_2} \ge \norm{\v x_1 - \v y} - \norm{\v x_2 - \v y}.
\]
Setting $\v y$ to be a closest point of $\v x_2$ on $\cM$, we have
$\norm{\v x_1 - \v x_2} \ge \norm{\v x_1 - \v y} - \dist(\v x_2, \cM) \ge \dist(\v x_1, \cM)- \dist(\v x_2, \cM)$.
Using symmetry, we also obtain the reverse inequality.
\end{proof}
The proof uses only that $\cM$ is closed. For a point cloud $\Gamma$, it shows that $\dist(\cdot,\Gamma)$ is $1$-Lipschitz, which is why the distance to $\Gamma$ recorded along a ray in~\cite{liu2024nearest} is continuous.

The following property of $\cP$ ensures that small perturbations along $\ell$ induce small displacements of the closest point on $\cM$.
\begin{proposition}[$\cP$ is continuous]\label{thm:curv:cp_cont}
If $\v x_n \to \v x$ within the reach of $\cM$ in $\bbR^d$, then $\cP \v x_n \to \cP \v x$ in $\bbR^d$.
\end{proposition}
\begin{proof}
First note that $\cP \v x_n$ is bounded uniformly. Indeed,
\begin{align*}
\norm{\v x_n - \cP \v x_n} &\le \norm{\v x_n - \cP \v x} \le \norm{\v x_n - \v x} + \norm{\v x - \cP \v x} = \norm{\v x_n - \v x} + \dist(\v x, \cM),\\
\implies \norm{\cP \v x_n} &\le \norm{\v x_n - \cP \v x_n} + \norm{\v x_n}
\le \dist(\v x, \cM) +\max_{n}  \norm{\v x_n - \v x} + \max_{n} \norm{\v x_n}.
\end{align*}
The last two terms are bounded since $\v x_n\to\v x$. Therefore, invoking the Bolzano--Weierstrass theorem, $\cP\v x_n$ has a convergent subsequence; denote it again by $\cP\v x_n$ for simplicity, and let its limit be $\v y\in\cM$. We can now identify $\v y$ using continuity of the distance:
\[ \norm{\v x - \v y} = \lim_{n\to\infty} \norm{\v x_n - \cP\v x_n} = \lim_{n\to\infty} \dist(\v x_n, \cM) = \dist(\v x, \cM).\]
By definition of $\cP$, the unique point in $\cM$ achieving this distance is $\cP\v x$, which shows $\v y = \cP\v x$. Thus every convergent subsequence of $\{\cP\v x_n\}$ converges to $\cP\v x$, implying $\cP\v x_n\to\cP \v x$.
\end{proof}

The reach also controls how fast the distance from $\v x$ to a point $\v z\in\cM$ grows as $\v z$ moves along $\cM$ away from the closest point $\cP\v x$.
\begin{proposition}[Quadratic growth within the reach]\label{thm:curv:quad_growth}
Let $\v x$ be within the reach of $\cM$, $\v y = \cP\v x$ and $\zeta = \norm{\v x-\v y}<\iota$. Then for every $\v z\in\cM$,
\begin{equation}\label{eqn:curv:quad_growth}
\norm{\v x-\v z}^2 \ge \zeta^2+\Big(1-\frac{\zeta}{\iota}\Big)\norm{\v z-\v y}^2 .
\end{equation}
\end{proposition}
\begin{proof}
By Federer's characterization of the reach~\cite[Thm.~4.18]{federer1959curvature},
\begin{equation}\label{eqn:curv:federer}\dist(\v z-\v y,\cT_{\v y}\cM)\le \frac{\norm{\v z-\v y}^2}{2\iota},\end{equation}
 for all $\v z\in\cM$. Since $\v x-\v y$ is a normal vector of length $\zeta$ at $\v y$, this gives \[\inpd{\v x-\v y}{\v z-\v y}\le\zeta\dist(\v z-\v y,\cT_{\v y}\cM)\le\zeta\frac{\norm{\v z-\v y}^2}{2\iota}.\] Hence
\[\norm{\v x-\v z}^2 = \zeta^2-2\inpd{\v x-\v y}{\v z-\v y}+\norm{\v z-\v y}^2 \ge \zeta^2+\Big(1-\frac{\zeta}{\iota}\Big)\norm{\v z-\v y}^2.\]
\end{proof}

\section{Investigating the induced curve}\label{sec:curv:induced_curve}

As outlined in~\cref{sec:curv:setup}, our approach generates a curve on $\cM$ induced by the closest point projection of a ray from the embedding space, and parameterized by the step size $h$ along the ray. The next step is to investigate this induced curve and understand how its curvature relates to the curvatures of the underlying manifold $\cM$.

\subsection{Geodesics contain information about the surface curvature}\label{sec:curv:geodesic}
Let $\cM \subset \bbR^{d}$ be the hypersurface in $\bbR^d$ and let $\alpha \subset \cM$ be a curve parameterized by \emph{arc length} $\cS$ where $\alpha(0) = \v y_0 \in \cM$, and denote by $\v n : \cM \to \bbR^{d}$ the surface normal vector field. At $\v y_0$, the tangent vector of the curve is $\alpha'(0)$, which is also tangent to $\cM$. Thus, $\inpd{\v n(\v y_0)}{\alpha'(0)}=0$, where $\inpd{\cdot}{\cdot}$ is the Euclidean inner product. Consequently, a local orthonormal frame at $\v y_0$ is given by the generalized Darboux frame
\begin{equation}\label{eqn:curv:darboux}
\big\{\alpha'(0),\,
\v v_2,\,\dots,\,\v v_{d-1},\,\v n(\v y_0)
\big\},
\end{equation}
where $\{\alpha'(0),\,\v v_2,\,\dots,\,\v v_{d-1}\}$ forms an orthonormal basis for the tangent space $\cT_{\v y_0}\cM$.
Under this frame, the principal normal $\alpha''(0)/\norm{\alpha''(0)}$ of $\alpha$ at $\v y_0$ lies in the two-dimensional plane spanned by $\v n$ and a unit tangent vector $\v g$ orthogonal to $\alpha'(0)$. Accordingly, $\alpha''(0)$ admits the decomposition
\begin{equation}\label{eqn:curv:curv_decomp}
  \alpha''(0) = \kappa_n(0)\v n + \kappa_g(0)\v g,
\end{equation}
where $\kappa_n$ and $\kappa_g$, denoting the \emph{normal curvature} and \emph{geodesic curvature}, are components of $\alpha''(0)$ on the normal space and the tangent space, respectively. The normal curvature depends only on the direction $\alpha'(0)$: it is $\sff_{\v y_0}(\alpha'(0),\alpha'(0))$ (\cref{def:curv:2nd_fund} and the discussion following it). More generally, for any parametrization of $\alpha$, the normal part of $\alpha''(0)$ is $\sff_{\v y_0}(\alpha'(0),\alpha'(0))\v n$ by Meusnier's theorem \eqref{eqn:curv:meusnier}, whether or not $\alpha$ is a geodesic.
The magnitude of $\alpha''(0)$ is the \emph{total curvature}:
\begin{equation}\label{eqn:curv:total_curvature}
    \kappa(0) = \norm{\alpha''(0)} = \sqrt{\kappa_n^2(0) + \kappa_g^2(0)}.
\end{equation}
A curve $\alpha$ is a \emph{geodesic} precisely when $\kappa_g \equiv 0$. Thus, if $\kappa_g(0) = 0$, then
$\alpha''(0)/\norm{\alpha''(0)}= \pm\v n(\v y_0)$,
and the total curvature $\kappa(0)$ equals $\abs{\kappa_n(0)}$, which coincides with the magnitude of the normal curvature of $\cM$ in the direction of $\alpha'(0)$.

Therefore, if one can obtain a curve $\alpha$ that is locally geodesic at $\alpha(0)=\v y_0$, then computing the curvature of this curve at $\alpha(0)$ reveals the normal curvature of the manifold at $\v y_0$.

\subsection{Relationships between the induced curve and geodesics}\label{sec:curv:hypersurface}
A natural question arises: is the induced curve obtained from the closest point projection a local geodesic? In this section, we analyze the explicit parameterized expression of the induced curve and study its geometric properties.

Consider the random ray $\{\ell(h)\defeq\v x_0 + h\v r\}$ and choose a coordinate frame such that $\v y_0 =\cP\v x_0= \v 0_d$ and $\ell(0)=\v x_0=(\v 0_{d-1},\, \zeta)$ with $\zeta = \norm{\v x_0-\v y_0}$.
By~\cref{thm:curv:cp_normal}, this is also a Darboux frame. We further align the basis so that moving from $\v x_0$ by $h\v r$ produces
\begin{equation}\label{eqn:curv:l_h}
    \ell(h)= (h,\, 0, \,\dots, \,0,\, \zeta + \gamma h),
\end{equation}
where $\gamma\in\bbR$ signifies the slope of the ray relative to $\cT_{\v y_0}\cM$ (we scale $\v r$ so that its tangential part has unit length).
We study the induced curve $\alpha(h) \defeq \cP(\ell(h))$. By~\cref{thm:curv:cp_cont}, it stays close to $\v y_0$ for small $h$, so a local representation of $\cM$ near $\v y_0$ suffices.
To express $\alpha(h)$ explicitly, one must solve the closest point condition
\begin{equation}\label{eqn:curv:cls_pt}
     \alpha(h) = \argmin_{\v y \in \cM} \norm{\ell(h) - \v y},
\end{equation}
which in turn requires an explicit local representation of $\cM$.
Within a sufficiently small neighborhood of $\v y_0$, the hypersurface $\cM$ under the above tangent-normal coordinate frame admits a Monge representation whose second-order truncation is the osculating paraboloid (\cref{thm:curv:monge})
\begin{equation}\label{eqn:curv:m_graph}
z(\v s) = \v s^T A \v s, \quad \v s \in \bbR^{d-1},
\end{equation}
where $A \in \bbR^{(d-1)\times(d-1)}$ is symmetric. By \cref{thm:curv:monge}, its eigenvalues $\lambda_i$ are half the principal curvatures of $\cM$ at $\v y_0$, taken with respect to the normal $\v e_d$ pointing toward $\v x_0$. Since no convexity assumption is imposed here, $A$ may be indefinite; the principal curvatures are bounded in absolute value by $1/\iota$ for $\iota$ the reach.
Since the curvatures of $\cM$ at $\v y_0$ are determined by the second-order local geometry, we first work with this quadratic model in a small neighborhood of $\v y_0$: we replace the hypersurface locally by its osculating paraboloid and derive the explicit parameterization of $\alpha$.

\begin{lemma}[Local parametrization of the induced curve]\label{thm:curv:parametrization}
Under the coordinate frame and the local graph model of $\cM$ described above, the closest point induced curve admits the local representation
\begin{equation}\label{eqn:curv:alpha_param}
\alpha(h) =\left ( \big(I - 2c(h) A \big)^{-1} h \v e_1,\; \cO(h^2) \right ),
\end{equation}
where $c(h) = \zeta + \gamma h + \cO(h^2)$ and $\v e_1 = (1,\v 0_{d-2})$.
\end{lemma}
\begin{proof}
Let $\alpha(h) = \cP(\ell(h)) = \big(\v s(h),\, z(\v s(h)) \big)$. Under the local graph model \eqref{eqn:curv:m_graph}, $\cM$ is described as the zero level set of $F(\v s,z)=z-\v s^T A\v s$, with $\nabla F(\v s, z) = (-2A\v s,\, 1)$ being the normal vector field on $\cM$.
On the other hand, \cref{thm:curv:cp_normal} implies $\ell(h) = \alpha(h) + c(h)\,\nabla F\big(\v s(h),\, z(h)\big)$ for some scalar function $c(h)$.
Using the local uniqueness of the closest-point projection within the reach, this relation becomes
\[
\begin{dcases}
   \v s(h) - 2c(h) A \v s(h) = h \v e_1, \\
   \v s(h)^TA\v s(h) + c(h) = \zeta + \gamma h,
\end{dcases}\implies \begin{dcases}
   \big(I - 2c(h) A\big)\v s(h) = h\v e_1, \\
   c(h) = \zeta + \gamma h - \v s(h)^T A\v s(h).
\end{dcases}
\]
The matrix $I - 2c(h)A$ is invertible provided $1/\big(2c(h)\big)$ is not an eigenvalue of $A$. Since the eigenvalues of $A$ are bounded by $1/(2\iota)$ while $c(h)\approx\zeta<\iota$, the product $2c(h)\lambda_i(A)$ satisfies $\abs{2c(h)\lambda_i(A)}<1$, ensuring invertibility within the reach; if $A$ is negative definite, it is invertible for every offset $\zeta>0$. Thus, the above equation admits a unique solution
\[
\v s(h) = \big(I - 2c(h)A\big)^{-1} h \v e_1.
\]
Because $A$ is symmetric, $I - 2c(h)A$ is also symmetric, and its invertibility ensures its eigenvalues remain bounded away from $0$. Hence $\normbig{\big(I - 2c(h)A\big)^{-1}}_2$ is bounded above, and
$\norm{\v s(h)}_2 \le \normbig{\big(I - 2c(h)A\big)^{-1}}_2\abs{h} = \cO(h)$.
By~\cref{thm:curv:cp_cont}, $z(h)\to 0$ as $h\to 0$, so $z(h)=\v s(h)^TA\v s(h)=\cO(h^2)$ for sufficiently small $h$, and therefore $c(h) = \zeta + \gamma h + \cO(h^2)$.
Substituting $\v s(h)$ and $z(h)$ completes the parametrization of $\alpha$.
\end{proof}

The first-order part of \cref{thm:curv:parametrization} depends only on the second-order geometry of $\cM$ at $\v y_0$, so it holds for $\cM$ itself and not only for its osculating paraboloid. In terms of the derivative of the closest point projection, it indicates that $\cD\cP(\v x_0)$ maps the tangential part $\v e_1$ of a ray to $(I-2\zeta A)^{-1}\v e_1$ and the normal $\v e_d$ to $\v 0$, that is,
\begin{equation}\label{eqn:curv:dP_hyper}
\cD\cP(\v x_0) = (I-\zeta S)^{-1}\Proj_{\cT_{\v y_0}\cM},\qquad S = 2A,
\end{equation}
where $S$ is the shape operator of $\cM$ at $\v y_0$ with normal pointing to $\v x_0$ (\cref{thm:curv:monge}). This identity holds at every offset within the reach, and it is the basis of the direct estimation in \cref{sec:curv:recon_direction}.

With the parameterized $\alpha(h)$ established, we now characterize when the induced curve is locally geodesic. The resulting criterion in~\cref{thm:curv:local_geodesic} not only identifies these cases, but also yields a practical way to generate local geodesics from the normal information of $\cM$.

\begin{theorem}\label{thm:curv:local_geodesic}
The curve $\alpha$ induced by the closest point projection of $\ell$ onto $\cM$ in the local graph model \eqref{eqn:curv:m_graph}, that is, onto the osculating paraboloid of $\cM$ at $\v y_0$, is a local geodesic at $\v y_0$ (that is, $\kappa_g(0)=0$) if and only if the direction $\v r$ of $\ell$ lies in the tangent space $\cT_{\v y_0}\cM$, or the tangential part $\v e_1$ of $\v r$ is a principal direction of $\cM$ at $\v y_0$. In particular, if $A$ is a multiple of the identity, as at every point of a sphere, every ray induces a local geodesic.
\end{theorem}
\begin{proof}
Let $\n$ denote the principal normal vector field of the curve $\alpha$, as defined in~\cref{sec:curv:geodesic}, where we also learn that $\alpha$ is a local geodesic at $\v y_0$ only when $\n(\v y_0)=\pm\v n(\v y_0)$.
Recall $\alpha$ is parameterized by $h$. By the Frenet--Serret formulas~\cite{jordan1874theorie}, one has
\begin{equation}\label{eqn:curv:frenet}
    \alpha''(h) = v'(h)\t(h) + \kappa(h)v^2(h)\n(h),
\end{equation}
where $\t$ is the unit tangent vector field and $v\defeq\dd{\cS}/\dd{h}$ for arc length $\cS$. On the other hand, by~\cref{thm:curv:parametrization},
\begin{equation}\label{eqn:curv:app}
  \alpha''(h) = \big(\v s''(h), \,z''(h)\big) = v'(h)\t(h) + \kappa(h)v(h)^2 \n(h).
\end{equation}
Under the Darboux frame in~\cref{sec:curv:hypersurface}, $\v n(\v y_0)=\v n(\v 0)=(\v 0_{d-1},\,1)$, and the tangent $\t(0)$ must have the final coordinate vanished. Indeed, using $\alpha'(0)=\big(\v s'(0), \,0\big)$ from~\cref{thm:curv:parametrization}, we have $v(0)\t(0) = \big(\v s'(0), \,0\big)$.
If $\alpha$ were a local geodesic at $\v y_0$, it is required that $\n(0)=\pm\v n(\v 0)$. From~\eqref{eqn:curv:app}, we conclude that
\begin{equation}\label{eqn:curv:geodesic_cond}
  v'(0)\,\t(0) = \big(\v s''(0),\, 0\big).
\end{equation}
Since $v' = \inpd{\alpha''}{\alpha'}/\norm{\alpha'}$, the left-hand side is the projection of $\alpha''(0)$ onto $\alpha'(0)$, so \eqref{eqn:curv:geodesic_cond} translates to
\begin{equation}\label{eqn:curv:proj_cond}
  \big(\v s''(0),\,0\big) = \Proj_{(\v s'(0),\,0)} \big(\v s''(0),\, z''(0)\big),
\end{equation}
which holds if and only if $\v s''(0)$ is parallel to $\v s'(0)$, i.e., $\v s''(0) = \beta\, \v s'(0)$ for some scalar $\beta$. A direct computation using the parametrization from~\cref{thm:curv:parametrization} yields
\begin{equation}\label{eqn:curv:s_derivatives}
\v s'(0) = (I - 2\zeta A)^{-1}\v e_1, \quad \v s''(0) = 4\gamma \,(I-2\zeta A)^{-1} A (I-2\zeta A)^{-1}\v e_1.
\end{equation}
If $\gamma = 0$, then $\v s''(0) = \v 0$ and the condition holds. If $\gamma\ne0$, write $\v e_1 = \sum_j a_j\v q_j$ in an orthonormal eigenbasis of $A$ with eigenvalues $\lambda_j$. Since $(I-2\zeta A)^{-1}$ and $A$ commute, the condition requires $\lambda_j/(1-2\zeta\lambda_j)$ to take the same value for all $j$ with $a_j\ne0$, and since $\lambda\mapsto\lambda/(1-2\zeta\lambda)$ is injective, $\lambda_j$ are constant for these $j$. This happens exactly when $\v e_1$ lies in one eigenspace of $A$, for instance $\v e_1 = \v q_j$ for one $j$, or equivalently when $\v e_1$ is a principal direction. In particular, if $A$ is a multiple of the identity, as on a sphere, this holds for every $\v e_1$. Otherwise, the induced curve is not a geodesic at $\v y_0$.
\end{proof}

\subsection{Convergence to geodesics}\label{sec:curv:convergence}
Next, we further investigate the behavior of the induced curve $\alpha$ when the generator ray $\ell$ is not tangent to $\cM$. Before proceeding, we state the following auxiliary result for our subsequent exploration.

\begin{lemma}[Bounds on the rotation angle induced by a definite matrix]\label{thm:curv:rotation_angle}
For a negative definite matrix $A\in\bbR^{d\times d}$ with condition number $\tau$, the rotation it induces on any nonzero $\v x\in\bbR^d$ satisfies
\begin{equation}\label{eqn:curv:rotation_bound}
  -1 \le \cos(\theta) = \frac{\inpd{\v x}{A\v x}}{\norm{\v x}\,\norm{A\v x}}  \le \frac{-2}{\sqrt{\tau}+\frac{1}{\sqrt{\tau}}},
\end{equation}
implying the rotation angle $\theta$ lies in $[\pi/2,\,\pi]$. An analogous bound with reversed sign holds for positive definite $A$.
\end{lemma}
\begin{proof}
The core idea is to optimize over the subspace spanned by $\v x$ and $A\v x$.
Since $A$ is symmetric, it admits the diagonalization $A=P^{T}\Lambda P$. Because the orthogonal matrix $P$ is surjective, the maximal rotation angle induced by $A$ is the same as that induced by $\Lambda$. Thus it suffices to analyze the diagonal case $\Lambda$. Since only the angle matters, we may assume $\norm{\v x}=1$, so
\begin{equation}\label{eqn:curv:cos_theta}
    \cos(\theta)=\frac{\v x^T\Lambda\v x}{\norm{\Lambda\v x}}.
\end{equation}
This expression depends only on the vectors $\v x$ and $\Lambda\v x$, which span a $2$-dimensional subspace of~$\bbR^d$. Let $B\in\bbR^{d\times 2}$ be the matrix whose columns $\v b_1,\v b_2$ form an orthonormal basis of $\operatorname{span}\{\v x,\,\Lambda\v x\}$. Then there exist $\v v_1,\v v_2\in\bbR^2$ such that $\v x=B\v v_1$ and $\Lambda\v x=B\v v_2$, and using $B^{T}B=I_2$, \eqref{eqn:curv:cos_theta} becomes $\v v_1^{T}\v v_2/\norm{\v v_2}$, the cosine of the angle between $\v v_1$ and $\v v_2$. Observe that $\v v_2=B^{T}\Lambda B \v v_1$. Set $N \defeq B^{T}\Lambda B\in\bbR^{2\times 2}$, which is symmetric, and let $N=Q^{T}\diag(\mu_1,\mu_2)Q$ be its eigendecomposition. Let $\lambda_{\min},\,\lambda_{\max}$ be the eigenvalues of $A$ of smallest and largest modulus. For any unit vector $\v w\in\bbR^2$, $\v w^TN\v w = (B\v w)^T \Lambda (B\v w)\in [\lambda_{\max},\,\lambda_{\min}]$, all negative, so $\mu_1, \mu_2\in[\lambda_{\max},\,\lambda_{\min}]$. With $\v u_1=Q\v v_1$ and $\v u_2=Q\v v_2$ we have $\v u_2 = \diag(\mu_1,\mu_2)\v u_1$. Because $\norm{\v u_1}=1$, we may parametrize $\v u_1 = (y,\,\sqrt{1-y^2})$ for some $y\in[-1,1]$. Then $\v u_2=(\mu_1y,\,\mu_2\sqrt{1-y^2})$, and therefore
\[
  \frac{\v v_1^{T}\v v_2}{\norm{\v v_2}} = \frac{\v u_1^{T}\v u_2}{\norm{\v u_2}}
  = \frac{(\mu_1-\mu_2)y^2 + \mu_2}{\sqrt{(\mu_1^2-\mu_2^2)y^2 + \mu_2^2}} \eqqcolon f(y^2).
\]
A direct derivative computation shows that the unique critical point of $f$ on $[0,1]$ is $z^* = \mu_2/(\mu_1+\mu_2)$, with
\[
  f(z^*) = \frac{2\mu_1\mu_2}{(\mu_1+\mu_2)\sqrt{\mu_1\mu_2}} = \frac{-2}{\sqrt{\mu_1/\mu_2} + \sqrt{\mu_2/\mu_1}},
  \qquad f(0)=f(1)=-1.
\]
By the AM--GM inequality, $f(z^*)\ge-1$, so $f(z^*)$ is the upper bound and $-1$ the lower bound. Without loss of generality, assume $0 >\mu_2\ge \mu_1$, and let $\tau = \lambda_{\max}/\lambda_{\min}$ be the condition number of $A$; then $1 \le \mu_1/\mu_2 \le \tau$. The function $g(t)=t+1/t$ is increasing on $[1,\infty)$, hence $f(z^*) \le -2/(\sqrt{\tau} + 1/\sqrt{\tau})$, which establishes the lemma.
\end{proof}
\begin{remark}\label{rmk:curv:rotation_angle}
We characterize the equality cases in~\cref{thm:curv:rotation_angle}.
\begin{itemize}
  \item The lower bound $-1$ is attained exactly when $\v x$ is an eigenvector of $A$. In such case the rotation angle is~$\pi$, meaning that $A\v x$ is a negative multiple of $\v x$ (a direction flip). In the special case where the condition number $\tau=1$, i.e., $\lambda_{\max}=\lambda_{\min}$, this holds for every $\v x$.
  \item The upper bound is attained when $\mu_1=\lambda_{\max}$ and $\mu_2=\lambda_{\min}$, and the low-dimensional coordinates of $\v x$ equal $\big( \pm \sqrt{\mu_2/(\mu_1+\mu_2)}, \;\pm \sqrt{\mu_1/(\mu_1+\mu_2)}\big)$ in the eigenbasis of $A$. The vector~$\v x$ achieving maximal rotation therefore lies in the span of $\v q_{\max}$ and $\v q_{\min}$, the eigenvectors corresponding to $\lambda_{\max}$ and $\lambda_{\min}$, and is aligned with $\v q_{\max} + \sqrt{\tau}\,\v q_{\min}$.
  \item An analogous argument yields $2/(\sqrt{\tau}+1/\sqrt{\tau}) \le \cos(\theta) \le 1$ for positive definite matrices, implying $\theta\in[0,\, \pi/2]$.
\end{itemize}
\end{remark}

When $\v x_0$ moves far away from $\cM$, it leaves the reach, and \cref{asm:curv:standing} no longer holds. As in \cref{thm:curv:local_geodesic}, we then work with the osculating paraboloid $z(\v s) = \v s^TA\v s$ of $\cM$ at $\v y_0$, and $\v y_0$ has to remain the closest point of $\v x_0$ on it. This needs the following assumption, through which any $\zeta>0$ is allowed.
\begin{assumption}\label{asm:curv:convex_side}
In the frame of \cref{sec:curv:hypersurface}, where $\v y_0 = \v 0$, $\v x_0 = (\v 0_{d-1},\zeta)$ and $\cM$ is locally the graph of $z(\v s) = \v s^TA\v s$, the matrix $A$ is negative definite. Equivalently, all principal curvatures of $\cM$ at $\v y_0$, taken with respect to the unit normal pointing toward $\v x_0$, are negative.
\end{assumption}
\begin{remark}\label{rmk:curv:indefinite_A}
The relevant matrix is $\cA\defeq(I - 2\zeta A)^{-1}A$, which governs the relation between $\v s'(0)$ and $\v s''(0)$; its eigenvalues are $\lambda_i/(1-2\zeta\lambda_i)$. Under \cref{asm:curv:convex_side}, these are negative for every $\zeta>0$, so $I-2\zeta A$ is invertible, and the closest point projection onto the osculating paraboloid is defined along the whole normal ray. Moreover, $\cA$ approaches $-(2\zeta)^{-1}I$ as $\zeta\to\infty$. If $A$ has a positive eigenvalue $\lambda$, then $I-2\zeta A$ is singular at the focal distance $\zeta = 1/(2\lambda)$, which yields an upper bound for the reach, and for larger $\zeta$ the point $\v y_0$ is no longer a closest point of $\v x_0$, and the limiting behavior does not hold. For saddle points and mixed curvature, the analysis therefore applies at finite $\zeta<\iota$ (\cref{thm:curv:parametrization,thm:curv:local_geodesic}), but not in the limit.
\end{remark}

Now we are ready to understand the behavior of the induced curve $\alpha$ for a non-tangent generator $\ell$.
\begin{theorem}[Convergence to geodesics]\label{thm:curv:r_convergence}
Let $\v x_0\in\bbR^d$ and $\v y_0=\cP\v x_0\in\cM$ be its closest point on $\cM\subset\bbR^d$, under \cref{asm:curv:convex_side}, and let $\cM$ be represented  by its osculating paraboloid \eqref{eqn:curv:m_graph} near $\v y_0$. If the ray $\ell\subset\bbR^d$ starting at $\v x_0$ has slope $\gamma\ne0$ relative to $\cT_{\v y_0}\cM$, as given in~\eqref{eqn:curv:l_h}, then its closest-point--induced curve $\alpha\subset\cM$ converges to a local geodesic at $\v y_0$ as $\zeta=\norm{\v x_0-\v y_0}\to\infty$, in the sense that the curvature $\kappa$ of $\alpha$ at $\v y_0$ converges to the magnitude of its normal curvature $\kappa_n$ according to
\begin{equation}\label{eqn:curv:geodesic_rate}
\abs{\kappa-\abs{\kappa_n}} = \cO(\gamma^2\zeta^{-2}),
\end{equation}
where $\kappa_n$ coincides with the curvature of the geodesic of $\cM$ at $\v y_0$ in the same direction as $\alpha$.

Furthermore, consider a coordinate frame in which $\v y_0$ is the origin and the osculating paraboloid of $\cM$ is given as $z(\v s)=\v s^T\Lambda\v s$ with $\Lambda=\diag(\lambda_1,\lambda_2,\dots,\lambda_{d-1})$. If the tangential part of the direction of $\ell$ is given by
$\v v=(v_1,\, v_2,\, \dots,\, v_{d-1})\in\cT_{\v y_0}\cM$,
then as $\zeta\to\infty$, the orientation of $\alpha$ at $\v y_0$ converges to
\begin{equation}\label{eqn:curv:limit_direction}
\Lambda^{-1}\v v=\bigg(\frac{v_1}{\lambda_1},\,\frac{v_2}{\lambda_2},\,\dots,\,\frac{v_{d-1}}{\lambda_{d-1}}\bigg)\in\cT_{\v y_0}\cM,
\end{equation}
which is also the direction of the limiting geodesic, with the corresponding curvature being
\begin{equation}\label{eqn:curv:limit_curvature}
\lim_{\zeta\to\infty}\kappa_n = 2\sum_{i=1}^{d-1}\frac{v_i^2}{\lambda_i}\Big/\sum_{i=1}^{d-1}\frac{v_i^2}{\lambda_i^2}.
\end{equation}
\end{theorem}
\begin{proof}
Fix the non-tangent ray $\ell$ and adopt the coordinate frame of~\cref{thm:curv:parametrization}, in which the tangential part of $\ell$ is along $\v e_1$. For finite $\zeta$, recall from~\eqref{eqn:curv:total_curvature} that the curvature $\kappa$ of $\alpha$ contains both the normal curvature and the geodesic curvature, and therefore differs from the curvature $\check\kappa$ of the true geodesic in that direction.
To obtain a closed form expression for $\check\kappa$, observe from~\cref{thm:curv:local_geodesic} that the local geodesic in the same initial direction can be identified explicitly by the induced curve of a tangent ray $\check\ell$ with the same offset $\zeta$ and the same tangent-plane projection as $\ell$. We use a check accent for all quantities associated with this tangent configuration.

By~\cref{thm:curv:parametrization}, the velocity $\v s'(0) = (I-2\zeta A)^{-1}\v e_1$ is independent of $\gamma$, so $\v s'(0)=\v{\check s}'(0)$, and $\check\alpha$ is a geodesic at $\v y_0$ in the direction $\v s'(0)$. Thus $\check\kappa(0) = \kappa_n$ is the normal curvature of $\cM$ at $\v y_0$ in this direction. For the tangent ray $\v{\check s}''(0)=\v 0$, which implies
\begin{align}
&\check v'(0)\t(0)=\Proj_{(\v{\check s}'(0),0)}\big(\v{\check s}''(0),\,\check z''(0)\big)=\v 0 \implies \check v'(0)=0, \nonumber\\
&\implies \check\alpha''(0) = \big(\v 0,\, \check z''(0)\big) = \check\kappa(0)\check v(0)^2\,\v n(\v y_0)
\implies\check\kappa(0) = \frac{\check z''(0)}{\v{\check s}'(0)\cdot\v{\check s}'(0)}.
\label{eqn:curv:kappaz}
\end{align}
Differentiating the graph representation yields
\begin{equation}\label{eqn:curv:zpp}
\check z''(h) = 2\big(A\v{\check s}'(h)\big)\cdot\v{\check s}'(h) +2\big(A\v{\check s}(h)\big)\cdot\v{\check s}''(h)
\implies \check z''(0) = 2\v{\check s}'(0)^TA\v{\check s}'(0).
\end{equation}
Substituting into~\eqref{eqn:curv:kappaz} gives
\[ \check\kappa(0) = \frac{2\v{\check s}'(0)^TA\v{\check s}'(0)}
{\v{\check s}'(0)\cdot\v{\check s}'(0)} = \frac{2\v e_1^T(I-2\zeta A)^{-1}A(I-2\zeta A)^{-1}\v e_1}
{\v e_1^T(I-2\zeta A)^{-2}\v e_1}. \]
Turning to the curve $\alpha(h)$, the Frenet identity~\eqref{eqn:curv:frenet} yields
$\norm{\v s''(0)}^2+z''(0)^2=v'(0)^2+\kappa(0)^2v(0)^4$,
since $\t$ and $\n$ are orthogonal. The tangential contribution is
\[
v'(0)^2 = \frac{\big(\alpha'(0)\cdot \alpha''(0)\big)^2}{\norm{ \alpha'(0)}^2} = \frac{\big(\v s'(0)\cdot\v s''(0)\big)^2}{\norm{\v s'(0)}^2}.
\]
Moreover, by~\eqref{eqn:curv:zpp}, $z''(0)$ only depends on $\v s'(0)$, and $\v s'(0) = \v{\check s}'(0)$, so $z''(0)^2=\check z''(0)^2=\check\kappa(0)^2v(0)^4$.
Combining terms and rearranging yields
\begin{align*}
\kappa(0)^2
&=
\frac{\norm{\v s''(0)}^2}{v(0)^4}+\check\kappa(0)^2 -\frac{(\v s'(0)\cdot\v s''(0))^2}{v(0)^6},
\\
\mathrm{Res}&\defeq\sqrt{\kappa(0)^2-\check\kappa(0)^2} = \frac{\sqrt{\norm{\v s''(0)}^2\norm{\v s'(0)}^2-(\v s'(0)\cdot\v s''(0))^2}}{v(0)^3}
=\frac{\norm{\v s''(0)}\sin\theta}{\norm{\v s'(0)}^2},
\end{align*}
where $\theta$ is the angle between $\v s'(0)$ and $\v s''(0)$.
To bound $\mathrm{Res}$, we bound the norms of $\v s''(0)$ and $\v s'(0)$. Let $\{\lambda_i\}_{i=1}^{d-1}$ denote the eigenvalues of $A$. Then $(I-2\zeta A)^{-1}$ has eigenvalues $(1-2\zeta\lambda_i)^{-1}$, and therefore
\[
\norm{\v s'(0)}^2 = \v e_1^T(I-2\zeta A)^{-2}\v e_1\ge \min_i\bigg(\frac{1}{1-2\zeta\lambda_i}\bigg)^2.
\]
On the other hand, $\v s''(0)$ is governed by the symmetric matrix $(I-2\zeta A)^{-1}A(I-2\zeta A)^{-1}$, whose factors share the same eigenvectors, so its eigenvalues are $\lambda_i/(1-2\zeta\lambda_i)^2$. Consequently,
\[
\norm{\v s''(0)}\le\max_i\absbig{\frac{4\gamma\lambda_i}{(1-2\zeta\lambda_i)^2}},
\qquad
\mathrm{Res}\le\max_i\absbig{\frac{4\gamma\lambda_i}{(1-2\zeta\lambda_i)^2}}\max_i\big(1-2\zeta\lambda_i\big)^2\sin\theta.
\]
It remains to bound $\sin\theta$. By~\cref{thm:curv:parametrization}, $\v s''(0)=4\gamma\cA\,\v s'(0)$ with $\cA=(I-2\zeta A)^{-1}A$, so $\sin\theta$ is controlled by the rotation exerted on $\v s'(0)$ by $\cA$. Under \cref{asm:curv:convex_side}, $A$ is negative definite, so the eigenvalues $\lambda_i/(1-2\zeta\lambda_i)$ of $\cA$ are negative for every $\zeta>0$, and $\cA$ is negative definite.
Invoking~\cref{thm:curv:rotation_angle}, the rotation angle $\theta'$ of $\cA$ satisfies $\cos\theta'\le -2/(\sqrt{\tau}+1/\sqrt{\tau})$, where $\tau = \tau(\zeta)$ is the condition number of $\cA$. Consequently,
\[
0\le\sin\theta'\le\frac{\sqrt{\tau}-\frac{1}{\sqrt{\tau}}}{\sqrt{\tau}+\frac{1}{\sqrt{\tau}}}=\frac{\tau-1}{\tau+1}
=\frac{\abs{\lambda_{\max}}-\abs{\lambda_{\min}}}{\abs{\lambda_{\max}}+\abs{\lambda_{\min}}+4\zeta\abs{\lambda_{\min}\lambda_{\max}}}
=\cO(\zeta^{-1}),
\]
with $\lambda_{\min},\lambda_{\max}$ the eigenvalues of $A$ of smallest and largest modulus. Multiplication by $4\gamma$ only changes the sign, so $\sin\theta=\sin\theta' = \cO(\zeta^{-1})$. Substituting, and using $1-2\zeta\lambda_i\approx2\zeta\abs{\lambda_i}$ for large $\zeta$,
\begin{equation}\label{eqn:curv:res_asymbound}
\mathrm{Res}\lesssim\frac{\abs{\gamma}\,\abs{\lambda_{\max}}\big(\abs{\lambda_{\max}}-\abs{\lambda_{\min}}\big)}
{\lambda_{\min}^2}\frac{1}{\zeta}= \abs{\gamma}\tau_A(\tau_A-1)\frac{1}{\zeta}=\cO(\zeta^{-1}),
\end{equation}
where $\tau_A = \lambda_{\max}/\lambda_{\min}$ is the condition number of $A$. Since $\kappa(0)^2-\check\kappa(0)^2 = \mathrm{Res}^2$ and $\kappa(0)\ge\abs{\check\kappa(0)}$, we conclude that $0\le\kappa(0)-\abs{\check\kappa(0)}\le\mathrm{Res}^2/(2\abs{\check\kappa(0)}) = \cO(\gamma^2\zeta^{-2})$, where $\abs{\check\kappa(0)}$ is bounded below because $\cM$ is nondegenerate at $\v y_0$.

Finally, the tangent direction of the induced curve is $\v s'(0)/\norm{\v s'(0)}$ with $\v s'(0)=(I-2\zeta A)^{-1}\v e_1$. Write $A=Q^T\Lambda Q$ with $\Lambda=\diag(\lambda_1,\dots,\lambda_{d-1})$; then $Q$ maps the frame of~\cref{thm:curv:parametrization} to the principal frame, and $Q\v e_1 = \v v$ is the tangential part of $\ell$ there. In the principal frame, the direction of $\alpha$ is
\[
Q\v s'(0) = (I-2\zeta\Lambda)^{-1}\v v = \diag\bigg(\frac{1}{1-2\zeta\lambda_1},\dots,\frac{1}{1-2\zeta\lambda_{d-1}}\bigg)\v v.
\]
For $\zeta$ sufficiently large so that $\abs{2\zeta\lambda_i}\gg 1$ for all $i$, $1-2\zeta\lambda_i\approx-2\zeta\lambda_i$, so the direction is effectively along $\Lambda^{-1}\v v$, independently of $\zeta$, and this approximation is exact as $\zeta\to\infty$. This direction coincides with that of the limiting geodesic. Correspondingly,
\[
\lim_{\zeta\to\infty}\check\kappa(0) = \frac{2\v e_1^TA^{-1}\v e_1}{\v e_1^TA^{-2}\v e_1}
= 2\frac{\v v^T\Lambda^{-1}\v v}{\v v^T\Lambda^{-2}\v v}
=2\sum_{i=1}^{d-1}\frac{v_i^2}{\lambda_i}\Big/\sum_{i=1}^{d-1}\frac{v_i^2}{\lambda_i^2}.
\]
\end{proof}
\begin{remark}
\Cref{thm:curv:r_convergence} implies that for various rays $\ell$ with different tangent-space projection directions, the corresponding limiting geodesics also differ in directions, provided the surface is non-degenerate.
\end{remark}
\begin{remark}
The argument of \cref{thm:curv:r_convergence} applies at every point $\v x\in\ell$ with the osculating paraboloid at $\cP\v x$ as long as \cref{asm:curv:convex_side} holds at $\cP\v x$. So, as $\zeta\to\infty$, $\alpha$ is asymptotically a geodesic along its whole length, to second order at each of its points.
\end{remark}

\Cref{thm:curv:local_geodesic,thm:curv:r_convergence} are statements about the osculating paraboloid. The curvature of the induced curve involves the second derivative of $\cP$, so it also depends on the third-order terms of $\cM$ at $\v y_0$.
\begin{corollary}[Third-order terms]\label{thm:curv:third_order}
Let $\cM$ be given near $\v y_0$ by $z(\v s) = \v s^TA\v s + C(\v s,\v s,\v s) + \cO(\norm{\v s}^4)$ with a symmetric trilinear form $C$, and let $G = (I-2\zeta A)^{-1}$. Then the induced curve satisfies
\begin{equation}\label{eqn:curv:s_derivatives_third}
\v s'(0) = G\v e_1,\qquad \v s''(0) = G\big(4\gamma A\v s'(0) + 6\zeta\, C(\v s'(0),\v s'(0),\cdot)\big).
\end{equation}
Under \cref{asm:curv:convex_side}, for every slope $\gamma$, its geodesic curvature at $\v y_0$ converges as $\zeta\to\infty$ to
\begin{equation}\label{eqn:curv:kappa_g_limit}
\kappa_g^\infty = \frac{3\,\norm{\Proj_{\v u^\perp}A^{-1}C(\v u,\v u,\cdot)}}{\norm{\v u}^2},\qquad \v u = A^{-1}\v e_1.
\end{equation}
\end{corollary}
\begin{proof}
With $\nabla z(\v s) = 2A\v s + 3C(\v s,\v s,\cdot) + \cO(\norm{\v s}^3)$, the closest point condition in the proof of \cref{thm:curv:parametrization} becomes $\v s(h)-c(h)\nabla z(\v s(h)) = h\v e_1$ with $c(h) = \zeta+\gamma h-z(\v s(h))$, so $c(0) = \zeta$ and $c'(0) = \gamma$. Differentiating once and twice at $h = 0$, where $\v s = \v 0$, gives
\[ \v s'(0) = G \v e_1 \text{ and } G^{-1}\v s''(0) = 4\gamma A\v s'(0)+6\zeta\, C(\v s'(0),\v s'(0),\cdot).\]
 Since $\alpha'(0) = (\v s'(0),0)$ and $\cT_{\v y_0}\cM = \bbR^{d-1}\times\{0\}$, the geodesic curvature at $\v y_0$ is
 \[\kappa_g(\v y_0) = \frac{\norm{\Proj_{\v s'(0)^\perp}\v s''(0)}}{\norm{\v s'(0)}^2}.\]
  Under \cref{asm:curv:convex_side}, $G = -(2\zeta)^{-1}A^{-1}+\cO(\zeta^{-2})$, so $\v s'(0) = -\v u/(2\zeta)+\cO(\zeta^{-2})$. The first term of $\v s''(0)$ is $4\gamma AG^2\v e_1 = \gamma\v u/\zeta^2+\cO(\zeta^{-3})$, which is parallel to $\v s'(0)$ up to $\cO(\zeta^{-3})$, and the second is $-\frac{3}{4\zeta^2}A^{-1}C(\v u,\v u,\cdot)+\cO(\zeta^{-3})$. Dividing by $\norm{\v s'(0)}^2 = \norm{\v u}^2/(4\zeta^2)+\cO(\zeta^{-3})$ gives
  \[\kappa_g(\v y_0) = \frac{3\norm{\Proj_{\v u^\perp}A^{-1}C(\v u,\v u,\cdot)}}{\norm{\v u}^2}+\cO(\zeta^{-1}),\] and letting $\zeta\to\infty$ gives the limit.
\end{proof}
So the third-order terms add to $\v s''(0)$ a term that is in general not parallel to $\v s'(0)$. For a tangent ray, the geodesic curvature at $\v y_0$ is then of order $\zeta\norm{C}$ for small $\zeta$, while it vanishes for the osculating paraboloid, and on the convex side the limit $\kappa_g^\infty$ vanishes only if $A^{-1}C(\v u,\v u,\cdot)$ is parallel to $\v u$. The limit $\zeta\to\infty$ gives a geodesic in every direction where $C$ vanishes, for example at every point of a sphere or of a circular cylinder. The normal part of $\alpha''(0)$ is not affected by $C$: it is $\sff_{\v y_0}(\alpha'(0),\alpha'(0))\v n$ for every parametrization (\cref{sec:curv:prelim_codim}). A curvature estimate that takes the curvature $\kappa$ of the induced curve for the normal curvature is therefore off by
\begin{equation}\label{eqn:curv:third_order_bias}\kappa-\abs{\kappa_n} = \frac{\kappa_g^2}{\kappa+\abs{\kappa_n}}\approx(\kappa_g^\infty)^2/(2\abs{\kappa_n})\end{equation} for large $\zeta$: this is quadratic in $C$, so it is small when the third-order terms are small compared with the curvature, but it does not decay with $\zeta$. Estimates using only the first derivative of $\cP$ or the normal part of $\alpha''(0)$ are not affected.

\section{Estimating the extrinsic curvatures of \texorpdfstring{$\cM$}{M}}\label{sec:curv:curv_compute}

The analysis above shows that, when $\ell \notin \cT_{\v y_0}\cM$, the discrepancy between the curvature of the induced curve $\alpha$ and the normal curvature of $\cM$ decays as $\cO(\gamma^2\zeta^{-2})$ where $\gamma$ denotes the relative slope and $\zeta$ the offset distance, for the osculating paraboloid and under \cref{asm:curv:convex_side}. In addition, \cref{thm:curv:parametrization} shows that a non-tangent ray introduces only an additive $\gamma h$ to $\zeta$. These observations yield a practical proxy for the normal curvature of $\cM$ with controllable error, provided that $\gamma$ is bounded, $h\ll 1$, and $\zeta\gtrsim 1$, up to a bias from the third-order terms of $\cM$, which does not decay with $\zeta$ (\cref{thm:curv:third_order}).

While $\zeta$ can be directly controlled as a hyperparameter when generating $\ell$, eliminating the effect of the slope $\gamma$ would require projection onto the tangent space, which often complicates the algorithm. Moreover, the normal $(\v x_0-\v y_0)/\zeta$ from \cref{thm:curv:cp_normal} is exact only for the exact projection, and it is less reliable when the projection is approximate, as on point clouds.
For practical implementation, it is therefore desirable that the random ray distribution is independent of the underlying manifold $\cM$. We choose $\v r \sim \mathrm{Unif}(\bbS^{d-1})$ for convenience.  Under this setting, we study the behavior of $\gamma$ systematically, and demonstrate both theoretically and numerically that it remains well-behaved.

\subsection{Behaviors of the slope \texorpdfstring{$\gamma$}{gamma} without normal information}\label{sec:curv:concentration}
In a low dimensional setting we have the following.
\begin{proposition}[Polynomial tail decay of slope distributions in low dimensions]\label{thm:curv:slope_tail}
Let the random direction be $\v r \sim \mathrm{Unif}(\bbS^{d-1})$. Then for $d=2, \,3$, the distribution of its slope $\gamma$ satisfies, as $t\to\infty$,
$\bbP(\gamma>t) = \cO\big(t^{-(d-1)}\big)$.
\end{proposition}
\begin{proof}
Let $\v n$ be the fixed unit normal, and define $Z\defeq \abs{\v r^T\v n}\in [0, 1]$. For $d=2$, $Z=\abs{\cos\theta}$, where $\theta\sim \mathrm{Unif}[0,2\pi)$ is the angle between $\v r$ and $\v n$. For $d=3$, $Z$ is uniformly distributed on $[0, 1]$ (Archimedes' theorem). The slope parameter is $\gamma\defeq Z/\sqrt{1-Z^2}$. Thus $\bbP(\gamma > t) = \bbP\big(Z> t/\sqrt{t^2+1}\big)$, and as $t\to\infty$,
\[
\bbP(\gamma > t) = \begin{dcases}
    \bbP\Big(\abs{\cos\theta}>t/\sqrt{t^2+1}\Big) = (2/\pi)\arctan(1/t) \sim 2/(\pi t),  & d=2,\\
    1- 1/\sqrt{1+1/t^2} = 1/(2t^2) +\cO(1/t^4),& d=3.
\end{dcases}
\]
\end{proof}

The polynomial tail decay of $\gamma$ in low dimensions already makes it possible to avoid explicit normal information. This decay is even more pronounced in higher dimensions, reflecting the concentration of measure phenomenon~\cite{boucheron2003concentration}. One representative result is due to P.~L\'evy.
\begin{lemma}[P.~L\'evy's concentration lemma, 1919~\cite{levy1922leccons,levy1951problemes}]\label{thm:curv:levy}
Let $\sigma$ denote the normalized uniform measure on the unit hypersphere $\bbS^{d-1}\subset\bbR^d$, and let $f:\bbS^{d-1}\to\bbR$ be a $c_f$-Lipschitz function with median $m_f$. Then $\sigma\big(\abs{f-m_f}\ge \varepsilon\big) \le 2\exp\big(-(d-2)\varepsilon^2/(2c_f^2)\big)$.
\end{lemma}
Following~\cref{thm:curv:levy}, we immediately arrive at the high probability concentration of the last coordinate.
\begin{corollary}\label{thm:curv:concentration_sphere}
For a uniformly random sample $\v r$ on the $(d-1)$-dimensional unit sphere, almost all mass concentrates near the equator. More precisely, for any fixed coordinate system in which $\v r = (r_1,\, r_2,\,\dots,\, r_d)$,
\begin{equation}\label{eqn:curv:concentration}
  \bbP\big(\abs{r_d}\ge \varepsilon\big) \le 2e^{-(d-2)\varepsilon^2/2}.
\end{equation}
\end{corollary}
\begin{proof}
Apply~\cref{thm:curv:levy} to $f(\v r)=r_d$, which is $1$-Lipschitz and, by symmetry, has median $0$.
\end{proof}
Thus, the probability that $M$ samples of such $r_d$ all lie within the band $\abs{r_d}\le\varepsilon$ is bounded below by
\[
 \bbP\big(\abs{r_d^i}\le\varepsilon,\; i=1,\,2,\,\dots,\,M\big)
  = \bbP(\abs{r_d}\le\varepsilon)^M
  \ge \Big(1-2e^{-\frac{d-2}{2}\varepsilon^2}\Big)^M
  \ge 1-2Me^{-\frac{d-2}{2}\varepsilon^2},
\]
where the last inequality is tight in the desired regime $Me^{-(d-2)\varepsilon^2/2} \ll 1$. To enforce a rapidly decaying failure probability with the dimension, one may prescribe
\begin{equation}\label{eqn:curv:p_lowbound}
\begin{gathered}
  2Me^{-\frac{d-2}{2}\varepsilon^2} < \frac{1}{d} \iff \varepsilon > \sqrt{\frac{2}{d-2}\ln(2dM)},\\
  \text{which implies}\quad\bbP\big(\abs{r_d^i}\le\varepsilon \text{ for all } i=1,\,\dots,\,M\big) > 1-\frac{1}{d}.
\end{gathered}
\end{equation}
The identity $\gamma = r_d/\sqrt{1-r_d^2}$ allows us to translate the probabilistic bounds on $r_d$ into explicit constraints on the slope $\gamma$: with probability at least $1-1/d$, all $\abs{r_d^i}\le\varepsilon$, and then $\abs{\gamma_i}\le\varepsilon/\sqrt{1-\varepsilon^2}\le\sqrt2\,\varepsilon$ if $\varepsilon^2\le1/2$. With $\varepsilon$ at the bound in \eqref{eqn:curv:p_lowbound}, this gives
\begin{equation}\label{eqn:curv:slope_bound}
  \abs{\gamma_i}\le 2\sqrt{\frac{\ln(2dM)}{d-2}}\quad\text{for all } i,
\end{equation}
as long as this bound is at most $1$. This worst-case bound is for the largest of the $M$ slopes, while a single slope is much smaller on average. The coordinate $r_d$ has the density $c_d(1-t^2)^{(d-3)/2}$ on $[-1,1]$, with $c_d$ being the normalizing constant (for $d = 3$ it is uniform, as in the proof of \cref{thm:curv:slope_tail}), and a direct integration gives
\begin{equation}\label{eqn:curv:slope_moments}
  \bbE\abs{\gamma} = \frac{2c_d}{d-2}\approx\sqrt{\frac{2}{\pi d}},\qquad \bbE[\gamma^2] = \frac{1}{d-3}\quad(d\ge4).
\end{equation}
The approximation holds for large $d$. Denote $\v r = \v X/\norm{\v X}$ with $\v X$ a standard Gaussian vector in $\bbR^d$. By the Gaussian annulus theorem~\cite[Thm.~3.1.1]{vershynin2018high}, $\norm{\v X}$ is close to $\sqrt d$, so $\sqrt d\,r_d$ is close to the standard Gaussian $X_d$, and $\gamma\approx r_d$, so $\bbE\abs{\gamma}$ is close to $\sqrt{2/(\pi d)}$, the mean of the absolute value of a Gaussian with variance $1/d$. By comparing these theoretically worst-case bounds with empirical values from numerical simulations, we assess sharpness and practical relevance of the lower bound~\eqref{eqn:curv:p_lowbound}.
\begin{figure}[htbp]
  \centering
  \includegraphics[width=0.9\textwidth]{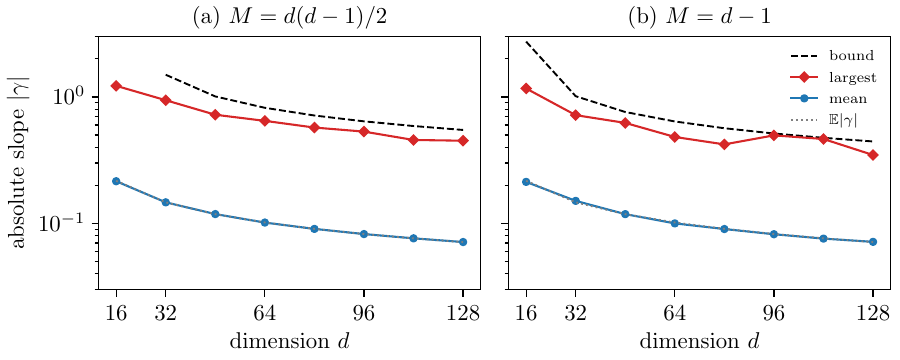}
  \caption{Comparison between theoretical bounds and simulated values.
  The bound $\varepsilon/\sqrt{1-\varepsilon^2}$ on the absolute slopes of $M$ uniformly random rays from~\eqref{eqn:curv:p_lowbound}, which holds with probability at least $1-1/d$ (dashed), against the largest and the mean absolute slope observed over $100$ draws of the $M$ rays, for (a) $M = d(d-1)/2$ and (b) $M = d-1$. Dotted: the mean $\bbE\abs{\gamma}$ of \eqref{eqn:curv:slope_moments}.}
  \label{fig:curv:gamma_compare}
\end{figure}

The comparison in~\cref{fig:curv:gamma_compare}, carried out for ambient dimensions ranging from $d=16$ to $d=128$, demonstrates that the theoretical bound provides sufficient and effective control over extreme slope values, despite the small probability of failure predicted by~\eqref{eqn:curv:p_lowbound}. In particular, $\gamma$ is well-controlled across the entire range of dimensions tested, and improves with increasing $d$, consistent with the concentration of measure phenomenon; the mean absolute slope follows \eqref{eqn:curv:slope_moments}.

\subsection{Reconstructing the second fundamental form}\label{sec:curv:recon_2nd}
Having established that non-tangent rays introduce only a controllable error, we now present practical methods inspired by the framework developed in~\cref{sec:curv:hypersurface}. The first one recovers extrinsic curvatures through reconstructing the second fundamental form $\sff_{\v y}$ of \cref{def:curv:2nd_fund}.
The bilinear form $\sff_{\v y}$ can be represented as a symmetric $(d-1)\times(d-1)$ matrix $\bar S_{\v y}$ under an orthonormal tangent basis. The \emph{principal curvatures} of $\cM$ at $\v y$ are precisely the eigenvalues of $\bar S_{\v y}$, and the \emph{normal curvature} along any unit $\bar{\v v}\in\cT_{\v y}\cM$ is given by $\sff_{\v y}(\bar{\v v}, \bar{\v v}) = \bar{\v v}^T\bar S_{\v y}\bar{\v v}$.

Since the normal curvatures are estimable using the induced-curve analysis developed in~\cref{sec:curv:hypersurface}, by collecting sufficiently many pairs of linearly independent unit directions $\{\bar{\v v}_i\}$ and normal curvatures $\sff_{\v y}(\bar{\v v}_i, \bar{\v v}_i)$, one may estimate $\bar S_{\v y}$ through a least squares formulation:
\begin{equation}\label{eqn:curv:lstsq}
    \hat{\bar S}_{\v y} = \argmin_{X} \sum_{i=1}^{M} \Big(\bar{\v v}_i^TX\bar{\v v}_i -  \sff_{\v y}(\bar{\v v}_i, \bar{\v v}_i)\Big)^2.
\end{equation}
A curvature measured along an induced curve has to be paired with the direction of that curve, not with the direction of the ray. By \cref{thm:curv:parametrization,thm:curv:r_convergence}, the curve induced by a ray with tangential part $\v e_1$ points in the direction of $(I-2\zeta A)^{-1}\v e_1$ at $\v y$, which differs from $\v e_1$ unless $\v e_1$ is a principal direction, and tends to $\Lambda^{-1}\v v$ as $\zeta\to\infty$. This direction can be measured without the normal: up to $\cO(h^2)$, it is the direction of the central difference $\cP(\v x_0+h\v r_i)-\cP(\v x_0-h\v r_i)$. These directions span $\cT_{\v y}\cM$ up to $\cO(h^2)$, so the least squares problem below can be solved on their span.

When using rays $\v r_i \notin \cT_{\v y}\cM$, the induced curvature no longer coincides with the normal curvature of $\cM$, and error propagation must be quantified. Fix $\v y\in\cM$; for induced curves $\alpha_i$ starting from $\v y$ generated by rays along directions $\v r_i$, $i=1,\, 2,\,\dots,\,M$, denote by $\v t_i\in\cT_{\v y}\cM$ the unit tangent of $\alpha_i$ at $\v y$, by $\kappa_{\alpha_i}$ its curvature at $\v y$, by $\kappa_{n,i}$ the corresponding normal curvature of $\cM$ in the direction $\v t_i$, and by $\omega_i$ the error between them. Under \cref{asm:curv:convex_side} all normal curvatures at $\v y$ are negative with respect to $\v n$, so $\kappa_{n,i} = -\abs{\kappa_{n,i}}$, and by~\cref{thm:curv:r_convergence}, for the osculating paraboloid,
\begin{equation}\label{eqn:curv:quadratic_form}
  -\kappa_{\alpha_i} = \kappa_{n,i} - \omega_i = \v t_i^T S_{\v y} \v t_i - \omega_i, \qquad 0\le\omega_i = \cO(\gamma_i^2 \zeta^{-2}),
\end{equation}
where $S_{\v y}\defeq Q \bar S_{\v y} Q^T$ denotes the extension of $\bar S_{\v y}$ to $\bbR^d$ under some orthonormal tangent basis $Q\in\bbR^{d\times (d-1)}$, such that $\v t_i^T S_{\v y} \v t_i = \bar{\v v}_i^T \bar S_{\v y}  \bar{\v v}_i$ for $\bar{\v v}_i = Q^T \v t_i$. $S_{\v y}$ still stores all information of the extrinsic curvatures. For $\cM$ itself, $\omega_i$ also contains a part due to the third-order terms of $\cM$ at $\v y$, which does not decrease with $\zeta$ (\cref{thm:curv:third_order}).

Let $\xi$ denote the vectorization of the symmetric matrix $S_{\v y}$, and let $\psi_i \defeq \v t_i \otimes \v t_i \in \bbR^{d^2}$ with $\otimes$ denoting the Kronecker product. Then~\eqref{eqn:curv:quadratic_form} converts to
\[
  -\kappa_{\alpha_i} = \psi_i^T \xi - \omega_i \implies
  -\bs\kappa_\alpha = \Psi \xi - \Omega,
\]
where $\bs\kappa_\alpha, \Omega\in\bbR^M$ collect the corresponding measurements and error terms, and $\Psi\in\bbR^{M\times d^2}$ has rows given by $\psi_i^T$. In practice, since only $\bs\kappa_\alpha$ is accessible, the least-squares estimator from~\eqref{eqn:curv:lstsq} now becomes
\begin{equation}\label{eqn:curv:lstsqsol}
    \hat\xi = -\Psi^{\dag} \bs\kappa_\alpha =  \Psi^{\dag}(\Psi\xi - \Omega) = \xi - \Psi^{\dag}\Omega,
\end{equation}
where the solvability of the system requires the $\v t_i\otimes\v t_i$ to span the symmetric matrices on $\cT_{\v y}\cM$, or equivalently, the design matrix is of full rank on that space. Reshaping $\hat\xi$ into the symmetric matrix yields the desired approximation $\hat S_{\v y}$ to $S_{\v y}$. The complete procedure is summarized in~\cref{alg:curv:recon_2nd}.

\begin{algorithm}[htbp]
\caption{Computing principal curvatures via reconstruction of the second fundamental form without exact normal}\label{alg:curv:recon_2nd}
\begin{algorithmic}
  \STATE \textbf{Input:} $\cM$: a hypersurface in $\bbR^d$ satisfying \cref{asm:curv:standing}, and \cref{asm:curv:convex_side} at the base point; $M$: the number of rays; $h$: the step size.
  \STATE Sample a point $\v x_0$ in the ambient space; find $\v y\defeq\cP\v x_0\in\cM$ and $\zeta = \norm{\v x_0-\v y}$.
  \STATE Generate random unit vectors $\{\v r_i\}_{i=1}^{M}$, $M\ge d(d-1)/2$.
  \FOR{$i=1:M$}
    \STATE Vary from $\v x_0$ along $\v r_i$ to generate the induced closest-point curve $\alpha_i$; estimate the curvature $\kappa_{\alpha_i}$ of $\alpha_i$ at $\v y$ and its unit tangent $\v t_i$.
  \ENDFOR
  \STATE Solve for $\hat S_{\v y}$, a symmetric matrix on the span of the $\v t_i$, by least squares from the $M$ equations $-\kappa_{\alpha_i} = \v t_i^T \hat S_{\v y} \v t_i$.
  \STATE Eigen-decompose $\hat S_{\v y}$; its first $d-1$ eigenvalues are the estimated principal curvatures $\hat\kappa_i$ at $\v y$.
\end{algorithmic}
\end{algorithm}

\begin{remark}\label{rmk:curv:meusnier}
The closest point pair also gives the normal $\v n = (\v x_0-\v y)/\zeta$ (\cref{thm:curv:cp_normal}) when the underlying target is smooth. With it, the normal curvature in the direction $\v t_i$ can be read off without any geodesic approximation: by \eqref{eqn:curv:curv_decomp} and Meusnier's theorem \eqref{eqn:curv:meusnier}, $\kappa_{n,i} = \inpd{\alpha_i''(0)}{\v n}/\norm{\alpha_i'(0)}^2$ for any parametrization of $\alpha_i$, with its sign. Using this quantity instead of $-\kappa_{\alpha_i}$ removes the error $\omega_i$ in \eqref{eqn:curv:quadratic_form}, including the part due to the third-order terms, and it needs neither \cref{asm:curv:convex_side} nor a large offset. The price is the use of the normal, which is exact for the exact projection but only approximate when the projection is approximate, as discussed at the start of this section.
\end{remark}

\subsection{Error analysis of~\texorpdfstring{\eqref{eqn:curv:lstsqsol}}{the least squares reconstruction}}\label{sec:curv:alg_recon_error}
The reconstruction error with random rays is bounded by
\begin{equation}\label{eqn:curv:error_reconstruct_B}
  \norm{\xi - \hat\xi} = \norm{\Psi^{\dag} \Omega} \le \norm{\Psi^{\dag}}\norm{\Omega},
\end{equation}
which requires bounds on both the error magnitude $\norm{\Omega}$ and the conditioning of the design matrix $\Psi$ on the symmetric matrices on $\cT_{\v y}\cM$.

By~\eqref{eqn:curv:slope_bound}, with probability at least $1-1/d$, $\abs{\gamma_i}\le 2\sqrt{\ln(2dM)/(d-2)}$ for all $i$. It follows that, with high probability,
\begin{equation}\label{eqn:curv:omega_bound}
  \norm{\Omega} \le \sqrt{M}\max_i\omega_i \lesssim \cO\Big(\frac{\ln(2dM)}{d\,\zeta^2}\Big)\sqrt{M}.
\end{equation}
On the other hand, $\norm{\Psi^{\dag}}$ is related to $\sigma_{\min}(\Psi)$, the smallest nonzero singular value of $\Psi$. Note each row of $\Psi$ is obtained via the Veronese mapping of degree $2$, $\nu_2(\v t) = \v t \otimes \v t$~\cite{harris1992algebraic,hartshorne2013algebraic}. For uniformly distributed directions, its covariance on the symmetric traceless matrices is $2/(k(k+2))$ times the identity, $k = d-1$ (\cref{thm:curv:isotropic}), and a matrix concentration argument gives $\norm{\Psi^\dagger} = \cO(1)$ once $M\gtrsim d^2\ln d$ (\cref{thm:curv:conditioning}). These two results and the proof of the following theorem are in \cref{sec:curv:app_alg1}.

\begin{theorem}[Errors of reconstructing the second fundamental form with random rays]\label{thm:curv:error_alg1}
Under \cref{asm:curv:convex_side}, let $\cM$ be represented near $\v y$ by its osculating paraboloid, and let $M\ge C(d-1)^2\ln d$ rays be drawn uniformly at random with offset distance $\zeta$, and let the curvature be nondegenerate at $\v y$. Then, with probability at least $1-2/d$, the eigenvalues $\hat\kappa_i$ of the second fundamental form estimated from~\eqref{eqn:curv:lstsqsol} satisfy
\begin{equation}\label{eqn:curv:alg1_frobenius}
  \sum_{i=1}^{d-1} \absmid{\hat\kappa_i-\kappa_i}^2 \le \norm{\hat S_{\v y}-S_{\v y}}_F^2
  \lesssim \cO\Big(\frac{\ln^2(d)}{\zeta^4}\Big),
\end{equation}
which leads to a bound on the individual error
\begin{equation}\label{eqn:curv:alg1_rate}
  \absmid{\hat\kappa_i-\kappa_i}\lesssim\cO\Big(\frac{\ln(d)}{\zeta^2}\Big),
  \quad i=1,\,2,\,\dots,\,d-1,
\end{equation}
where $\kappa_i$'s are the corresponding principal curvatures, and the constants depend on the condition number $\tau_A$ of the matrix $A$ of this paraboloid and on $\min_i\abs{\kappa_i}$.
\end{theorem}

\begin{remark}[The error bound as a function of $d$]\label{rmk:curv:alg1_rate}
\Cref{thm:curv:error_alg1} uses only $0\le\omega_i\le\max_j\omega_j$, and from this alone its rate in $d$ cannot be improved. In practice, the errors are of order $1/(d\zeta^2)$ because of the structure of the $\omega_i$. By \cref{thm:curv:parametrization}, $\v s''(0) = 4\gamma\,GAG\v e_1$ is linear in the slope, so the geodesic curvature of an induced curve is $\kappa_g = 4\abs\gamma\,q(\v u)$, where $\v u$ is the direction of the tangential part of the ray and $q$ does not depend on $\gamma$. Hence
\begin{equation}\label{eqn:curv:omega_structure}\omega = \kappa-\abs{\kappa_n}\approx\kappa_g^2/(2\abs{\kappa_n}) = \gamma^2\varphi(\v u) \text{ with }\varphi = 8q^2/\abs{\kappa_n}.\end{equation}
For uniform rays, $\gamma$ is independent of $\v u$, so the mean of $\omega_i$ for a given direction is $\bbE[\gamma^2]\varphi(\v u) = \varphi(\v u)/(d-3)$ by \eqref{eqn:curv:slope_moments}, while the theorem uses $\max_j\gamma_j^2\approx4\ln(2dM)/d$. The function $\varphi$ is of order $\zeta^{-2}$ and is a smooth function of some quadratic forms in $\v u$, so least squares produces an $\cO(\zeta^{-2})$ matrix, which is diagonal in the principal frame by the symmetry $u_l\mapsto-u_l$. Such an error is spread over all directions and cannot concentrate on one principal curvature, as the worst case of the theorem does. Finally, the deviations $(\gamma_i^2-\bbE[\gamma^2])\varphi(\v u_i)$ are independent and have mean zero, so their effect on the least squares solution averages out like $1/\sqrt M$.
\end{remark}

Finally, this approach requires solving a linear system of size $M\times d^2$ with $M = \cO(d^2\ln d)$, leading to an expensive cost of roughly $\cO(d^6\ln d)$. Moreover, the number of rays, and thus of closest point projections, need to scale with $d^2\ln d$. These challenges motivate the development of more efficient alternatives.

\subsection{Direct estimation of principal curvatures}\label{sec:curv:recon_direction}
The requirement of $\cO(d^2\ln d)$ rays arises from explicitly reconstructing all individual entries of $S_{\v y}$, which has $d(d-1)/2$ independent entries. Recall from~\cref{thm:curv:r_convergence} that \emph{the deviation of an induced curve direction is intrinsically linked to the principal curvatures} of $\cM$. Exploiting this connection enables estimating principal curvatures directly without reconstructing $S_{\v y}$.

Let $(\v x_0,\v y_0)$ be a closest-point pair with $\v x_0\in\ell_i$, where $\ell_i$ denotes the $i$-th random ray starting from $\v x_0$, and $\v y_0\in\cM$. Before proceeding, we clarify the expression for the orthogonal projection of $\ell_i$ onto $\cT_{\v y_0}\cM$ under different coordinate frames.
\begin{enumerate}
  \item \emph{Frame--1 (aligned local frame).}
  Under the local coordinate frame in which $\v y_0$ is the origin and $\ell_i$ lies in the $x_1$-$x_d$ plane, the projection is along $\v e_1$.
  \item \emph{Frame--2 (graph coordinate frame).}
  Under the local coordinate frame in which $\v y_0$ is the origin and the osculating paraboloid of $\cM$ is $z(\v s) = \v s^T \Lambda \v s$ with $\Lambda = \diag(\lambda_1,\dots,\lambda_{d-1})$ diagonal, i.e., the frame of principal directions, the projection of $\ell_i$ is along
  $\v v_i = (v_1,\,v_2,\,\dots,\,v_{d-1},0) = Q_i \v e_1$,
  for an orthogonal matrix $Q_i\in O(d)$ depending on the direction of $\ell_i$.
  \item \emph{Frame--3 (global coordinate frame).}
  In the global coordinate frame where $\cM$ is embedded in $\bbR^d$, the direction of the projection is again obtained through an orthogonal transformation
  $\v u_i = \cF(\v v_i) = P \v v_i$,
  where $P\in O(d)$ depends only on $\cM$ and the choice of $\v y_0$.
\end{enumerate}
In practice, computations are carried out in Frame--3. Given a collection of projected directions $\{\v u_j\}_j$ in the tangent space, the corresponding directions in Frame--2 are obtained via $\v v_j = \cF^{-1}(\v u_j) = P^T \v u_j$.
We now present the main result enabling direct extraction of the principal curvatures.
\begin{theorem}\label{thm:curv:recon_direction}
Fix a step size $h\ll 1$ and $\v x_0\in\bbR^d$, whose closest point on $\cM$---satisfying \cref{asm:curv:standing}---is denoted by $\v y_0 = \cP\v x_0$. Let $\zeta=\norm{\v x_0-\v y_0}$, let $\{\v u_i\}_{i=1}^{d-1}\subset\bbR^{d-1}$ be an orthonormal basis of $\cT_{\v y_0}\cM$, and $\{\bar{\v u}_i\}_{i=1}^{d-1}$ be its natural embeddings in $\bbR^d$. Define $\{\v w_i\}_{i=1}^{d-1}$ to be the projections of displacement along the induced curve onto $\cT_{\v y_0}\cM$, i.e.,
\begin{equation}\label{eqn:curv:w_tangent}
  \v w_i = \Big(\cP(\v x_0 + h\bar{\v u}_i)-\v y_0\Big)\Big|_{\cT_{\v y_0}\cM},
\end{equation}
and $U,W\in\bbR^{(d-1)\times(d-1)}$ be matrices whose columns are given by $\{\v u_i\}_{i=1}^{d-1}$ and $\{\v w_i\}_{i=1}^{d-1}$, respectively. Then the principal curvatures $\{\kappa_i\}_{i=1}^{d-1}$ of $\cM$ at $\v y_0$ may be approximated from the eigenvalues $\{\mu_i\}_{i=1}^{d-1}$ of $WU^T$ via
\begin{equation}\label{eqn:curv:extract_kappa}
  \hat\kappa_i = \left(1 - \frac{h}{\mu_i}\right)\Big/\zeta,
\end{equation}
with approximation error
\begin{equation}\label{eqn:curv:direct_rate}
  \hat\kappa_i = \big(1+\cO(h^2/\zeta)\big)\kappa_i.
\end{equation}
\end{theorem}
\begin{proof}
Let $\tilde \zeta = \norm{\v x_0-\v y_0} + \cO(h^2) = \zeta + \cO(h^2)$. \Cref{thm:curv:parametrization} states that, under Frame--1, the tangential component of the induced curve $\alpha$ generated by the direction $\v e_1$ satisfies
$\alpha(h)\big|_{\cT_{\v y_0}\cM} = (I-2\tilde \zeta A)^{-1} h\v e_1$.
Under Frame--2, the corresponding expression for direction $\v v_i=Q_i\v e_1$ becomes
$\alpha(h)\big|_{\cT_{\v y_0}\cM} = (I-2\tilde \zeta \Lambda)^{-1} h\v v_i$.
Expression in Frame--3 can be transformed to Frame--2 via the orthogonal map $\cF$, so we obtain
\[
  P^T \v w_i = (I-2\tilde \zeta \Lambda)^{-1} hP^T\v u_i\implies \v w_i = P(I-2\tilde \zeta \Lambda)^{-1} h\,P^T\v u_i.
\]
Collecting $\{\v w_i\}_i$ and $\{\v u_i\}_i$ into $W$ and $U$, respectively, and enforcing $\{\v u_i\}_i$ to be an orthonormal basis yields
\begin{equation}\label{eqn:curv:key_alg2}
  WU^T = P(I-2\tilde \zeta \Lambda)^{-1} hP^T.
\end{equation}
\Cref{eqn:curv:key_alg2} gives $\mu_i = h/(1-2\tilde \zeta \lambda_i)$. Since the osculating paraboloid $z(\v s) = \v s^TD\v s$ is diagonalized in Frame--2, it follows that
\[
  \kappa_i = 2\lambda_i = \frac{1-h/\mu_i}{\tilde \zeta}\approx \frac{1-h/\mu_i}{\norm{\v x_0-\v y_0}} = \hat\kappa_i,
\]
since we can only compute $\hat\kappa_i$ using $\zeta=\norm{\v x_0-\v y_0}$. Therefore, $(\tilde \zeta/\zeta)\kappa_i=\hat\kappa_i \implies \hat\kappa_i = \big(1+\cO(h^2/\zeta)\big)\kappa_i$.
\end{proof}

By \eqref{eqn:curv:dP_hyper}, the identity \eqref{eqn:curv:key_alg2} says that $WU^T$ is $h\cD\cP(\v x_0)$ up to $\cO(h^2)$. Only first derivatives of $\cP$ enter, so the third-order terms of $\cM$ play no role here, and $\v x_0$ may lie on either side of $\cM$: no convexity is needed, and saddle points are treated like any other point, since $(I-\zeta S)^{-1}$ exists for every $\zeta<\iota$ whatever the signs of the principal curvatures.

To compute~\eqref{eqn:curv:extract_kappa}, one can construct $W$ and $U$ through closest point queries, then the eigenvalues $\{\mu_i\}_{i=1}^{d-1}$ of $WU^T$ are readily available via eigen-decomposition. The only caveat is that so far columns of $W$ and $U$ are intrinsically defined in $\cT_{\v y_0}\cM$. When the normal is not used, for instance because it is only approximate, the reconstruction must instead proceed directly in the ambient $\bbR^d$. Motivated by the robustness to the slopes of random rays shown in~\cref{thm:curv:error_alg1}, we examine how dispensing with the tangent projection affects the present analysis.

\subsection{Errors of direct extraction of principal curvatures without normal}\label{sec:curv:alg2_error}
With a slight abuse of notation, let $U,W\in\bbR^{d\times d}$ denote matrices defined analogously to those in~\cref{thm:curv:recon_direction}, but constructed without tangent projection so that $U$ collects $d$ orthonormal ray directions in $\bbR^d$. The following theorem provides a precise description of the resulting approximation error.

\begin{theorem}\label{thm:curv:error_alg2}
When non-tangent rays with step size $h$ are used, the principal curvatures $\{\hat\kappa_i\}_{i=1}^{d-1}$ reconstructed from the leading $d-1$ eigenvalues of $WU^T$ approximate the true principal curvatures $\{\kappa_i\}_{i=1}^{d-1}$ with relative error
\begin{equation}\label{eqn:curv:alg2_rate}
  \hat\kappa_i=\left(
    1+\cO\left(\frac{h}{\zeta}\sqrt{\frac{\ln(d)}{d}}\right)
  \right)\kappa_i,
\qquad i=1,\dots,d-1,
\end{equation}
with probability at least $1-1/d$, as $h\to0$.
\end{theorem}
\begin{proof}
The core idea is that, without tangent space projection, the quantities in the final entries of the columns in $U$ and $W$ remain small under Frame--2, resulting in the matrix product $WU^T$ being dominated by its leading $(d-1)\times(d-1)$ block, which contains the desired eigenvalues.
Let $\bar U=P^TU$ and $\bar W=P^TW\in\bbR^{d\times d}$ denote the Frame--2 representations of $U$ and $W$. Then
\[
  WU^T = P\bar W\bar U^TP^T = P\sum_{i=1}^d \bar{\v w}_i\bar{\v u}_i^TP^T,
\]
so analyzing $WU^T$ reduces to analyzing the outer products $\bar{\v w}_i\bar{\v u}_i^T$. By~\cref{thm:curv:parametrization},
\[
  \bar{\v u}_i = [\bar{\v v}_i,\,\gamma_i]^T, \quad
  \bar{\v w}_i = \Big[\big(I-2c_i(h)\Lambda\big)^{-1}h\bar{\v v}_i,\; \cO(h^2)\Big]^T,
\]
where $\Lambda$ is the diagonal matrix of half the principal curvatures at $\v y_0\in\cM$ and $c_i(h) = \zeta+\gamma_ih+\cO(h^2)$. Therefore,
\[\bar{\v w}_i\bar{\v u}_i^T=
  \begin{bmatrix}
    \big(I-2c_i(h)\Lambda\big)^{-1}h\bar{\v v}_i\bar{\v v}_i^T
    &
    \gamma_i\big(I-2c_i(h)\Lambda\big)^{-1}h\bar{\v v}_i
    \\
    \cO(h^2)\bar{\v v}_i^T
    &
    \gamma_i\cO(h^2)
  \end{bmatrix}.
\]
Since $U$ is orthogonal, so is $\bar U$, and hence $I_d = \bar U\bar U^T$ gives
$\sum_{i=1}^d \bar{\v v}_i\bar{\v v}_i^T = I_{d-1}$ and $\sum_{i=1}^d \gamma_i\bar{\v v}_i = \v 0_{d-1}$.
At first order in $h$, the offsets $c_i(h)$ can be replaced by $\zeta$, and therefore
\[
  \sum_{i=1}^d \bar{\v w}_i\bar{\v u}_i^T=
  \begin{bmatrix}
    \big(I-2\zeta \Lambda\big)^{-1}h &  \v 0_{d-1}
    \\
    \cO(dh^2)\v 1_{d-1}^T  &  \cO\big(\ln(d)h^2\big)
  \end{bmatrix},
\]
where the bound on the last entry uses~\cref{thm:curv:concentration_sphere}. By Laplace expansion, the characteristic polynomial factorizes as
\[
  \det\Big(\big(I-2\zeta \Lambda\big)^{-1}h - \mu I_{d-1}\Big) \Big(\cO\big(\ln(d)h^2\big)-\mu\Big)=0.
\]
Hence, the top $d-1$ eigenvalues of $\sum_{i=1}^d \bar{\v w}_i\bar{\v u}_i^T$ coincide with those of $\big(I-2\zeta \Lambda\big)^{-1}h$, while the smallest eigenvalue is of order $\cO\big(\ln(d)h^2\big)$, provided $h$ is sufficiently small. Finally, the slopes enter at the next order. Since $c_j(h) = \zeta+\gamma_jh+\cO(h^2)$, the top-left block is
\[(I-2\zeta \Lambda)^{-1}h+2h^2\Lambda(I-2\zeta \Lambda)^{-2}\Theta+\cO(h^3)\]
 with $\Theta\defeq\sum_j\gamma_j\bar{\v v}_j\bar{\v v}_j^T$, and $\norm\Theta\le\max_j\abs{\gamma_j}$ because $\sum_j\bar{\v v}_j\bar{\v v}_j^T = I_{d-1}$. By the Bauer--Fike theorem~\cite{bauer1960norms}, $\mu_i = h/(1-2\zeta_i\lambda_i)$ with an effective offset $\zeta_i = \zeta+\cO(\max_j\abs{\gamma_j}h)$, which plays the role of $\tilde\zeta$ in the proof of \cref{thm:curv:recon_direction}. So similarly, $\hat\kappa_i = (1-h/\mu_i)/\norm{\v x_0-\v y_0}$, and applying~\eqref{eqn:curv:slope_bound} with $M = d$, we conclude that, with probability at least $1-1/d$,
\[
  \hat\kappa_i = \frac{\zeta_i}{\zeta}\kappa_i = \frac{\zeta+\cO(\max_j\abs{\gamma_j}h)}{\zeta}\kappa_i \implies   \hat\kappa_i=\left(1+\cO\left(\frac{h}{\zeta}\sqrt{\frac{\ln(d)}{d}}\right)
  \right)\kappa_i.
\]
\end{proof}
Compared with the error bound in~\cref{thm:curv:recon_direction}, the use of non-tangent rays adds a relative error of order $(h/\zeta)\sqrt{\ln(d)/d}$, which is first order in $h$ but decreases as $d$ grows. We summarize the entire procedure in~\cref{alg:curv:recon_direction}.

\begin{algorithm}[htbp]
\caption{Compute principal curvatures via~\cref{thm:curv:recon_direction} without normal information}\label{alg:curv:recon_direction}
\begin{algorithmic}
  \STATE \textbf{Input:} $\cM$: a hypersurface in $\bbR^d$ satisfying \cref{asm:curv:standing}; $h$: step size used to move along each ray.
  \STATE Sample a random point $\v x_0$ in the ambient space $\bbR^d$. Find the closest point $\v y_0\defeq\cP\v x_0\in\cM$.
  \STATE Record $\zeta=\norm{\v x_0-\v y_0}$; generate orthonormal $\{\v u_i\}_{i=1}^{d}$ in $\bbR^d$.
  \FOR {$i=1:d$}
    \STATE Find the closest point $\v y_i\defeq\cP(\v x_0+h\v u_i)\in\cM$; record $\v w_i=\v y_i-\v y_0$.
  \ENDFOR
  \STATE Construct $W,U\in\bbR^{d\times d}$ whose columns are $\{\v w_i\}_{i=1}^d$ and $\{\v u_i\}_{i=1}^d$, respectively.
  \STATE Eigen-decompose $WU^T$ to obtain the top $d-1$ eigenvalues $\{\mu_i\}_{i=1}^{d-1}$; compute $\hat\kappa_i$ via~\eqref{eqn:curv:extract_kappa}.
\end{algorithmic}
\end{algorithm}

We next conduct numerical experiments to verify the convergence behavior predicted by~\cref{thm:curv:recon_direction} and~\cref{thm:curv:error_alg2}. Specifically, we apply~\cref{alg:curv:recon_direction} to the unit $2$-sphere embedded in $\bbR^3$, using $50$ random starting points $\v x_0$ and orthonormal frames of rays and varying key parameters. The results are shown in~\cref{fig:curv:sphere_surface_compare}.
\begin{figure}[htbp]
  \centering
  \includegraphics[width=0.9\textwidth]{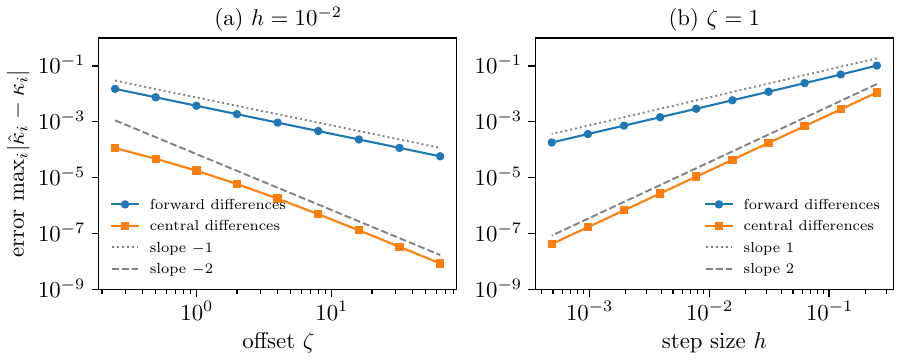}
  \caption{Performance of~\cref{alg:curv:recon_direction} on the unit $2$-sphere.
  Largest error of the two principal curvatures, median over $50$ random base points and orthonormal frames of rays, with the forward differences of \cref{alg:curv:recon_direction} and with central differences. (a) Against the offset $\zeta$ for $h = 10^{-2}$. (b) Against the step size $h$ for $\zeta = 1$.}
  \label{fig:curv:sphere_surface_compare}
\end{figure}

With the forward differences of \cref{alg:curv:recon_direction}, the error decreases like $h/\zeta$, the rate of \cref{thm:curv:error_alg2} (\cref{fig:curv:sphere_surface_compare}). With central differences, $\v w_i = \big(\cP(\v x_0+h\v u_i)-\cP(\v x_0-h\v u_i)\big)/2$, the slope term $\gamma_ih$ of the effective offset changes sign with $h$ and cancels, so the error decreases like $h^2$, as well as $\zeta^{-2}$ for large $\zeta$, at the price of $d$ more projections.

Finally, we emphasize that the number of directions used in~\cref{alg:curv:recon_direction}, $d$, scales linearly with the ambient dimension, representing a substantial reduction compared to the $\cO(d^2\ln d)$ rays required by estimating the second fundamental form in~\cref{sec:curv:recon_2nd}, thereby avoiding a large number of rays and a large least squares problem in high dimensions. Moreover, the reduction directly translates into fewer closest-point queries, leading to significant computational savings, particularly in high dimensional settings. Last but not least, we also avoid the $\cO(d^6\ln d)$ complexity of solving the linear system, requiring only $\cO(d^3)$ cost from an eigendecomposition and a matrix-matrix product between $\cO(d)\times\cO(d)$ matrices instead.

\subsection{Experiments in higher dimensions}\label{sec:curv:high_dim}
To test both algorithms beyond three dimensions, we use the hypersurface $\cM\subset\bbR^d$ given by the graph of
\begin{equation}\label{eqn:curv:test_function}
  f(\v s) = \frac12\v s^TD\v s+\frac14\norm{\v s}^4,
\end{equation}
where $D$ is a diagonal matrix with entries equally spaced in $[0.5,2]$. The Hessian of $f$ is $D+\norm{\v s}^2I+2\v s\v s^T$, so away from the origin the principal directions are not the coordinate directions and the third-order terms of $\cM$ do not vanish; at the origin they vanish, since $f$ is even. We compute the closest point projection exactly, up to rounding, using second-order Newton's method for $\min_{\v s}\norm{\v x-(\v s,f(\v s))}^2$, and the true principal curvatures using the derivatives of $f$. \Cref{fig:curv:high_dim} shows the results.
\begin{itemize}
  \item \Cref{alg:curv:recon_direction}: at five random base points along the level set $\norm{\v s} = 0.3$ and $h = 10^{-4}$ on both sides of the graph, the largest error over all $d-1$ principal curvatures stays below $2.1\times10^{-4}$ from $d = 8$ to $d = 128$, and it decreases slowly as $d$ grows (\cref{fig:curv:high_dim}(a)) since the slopes of the rays shrink like $1/\sqrt d$ (\cref{thm:curv:error_alg2}), and the errors due to the third-order terms average out over the $d$ ray directions. It decreases linearly in $h$ (\cref{fig:curv:high_dim}(b)). If one closest point projection costs $C_\cP$, the total cost of one estimate is $\cO(d\,C_\cP+d^3)$.

  \item \Cref{alg:curv:recon_2nd} with $d = 8$ and $M = 276$ rays, at the convex side of the origin, the error decreases like $\zeta^{-2}$ for large offset $\zeta$, as \cref{thm:curv:error_alg1} predicts. At a base point along the level set $\norm{\v s} = 0.3$, it stays between $0.05$ and $0.07$ for different values of $\zeta$, which shows the bias of \cref{thm:curv:third_order}; it grows slowly with $\zeta$ as the geodesic curvature of the induced curves approaches its limit $\kappa_g^\infty$. With the normal curvature computed according to \cref{rmk:curv:meusnier}, the error is below $2\times10^{-7}$ at both points and for every $\zeta$ (\cref{fig:curv:high_dim}(c)). Its rise at $\norm{\v s} = 0.3$ for $\zeta\ge16$ comes from the amplification of rounding errors: the second order finite difference divides them by $h^2$, while the normal curvature divides by $\norm{\alpha_i'(0)}^2$, which decreases like $\zeta^{-2}$. At the origin, where $\v y_0 = \v 0$ and the rounding errors of the projected points are much smaller, the rise does not occur.
\end{itemize}
Dense point clouds of these hypersurfaces are out of reach, since a point cloud with spacing $\varepsilon$ on a hypersurface in $\bbR^d$ needs of the order of $\varepsilon^{-(d-1)}$ points. Data in high dimensions concentrate instead near manifolds of low dimension and high codimension (\cref{sec:curv:intro}), which we treat next.

\begin{figure}[htbp]
  \centering
  \includegraphics[width=0.9\textwidth]{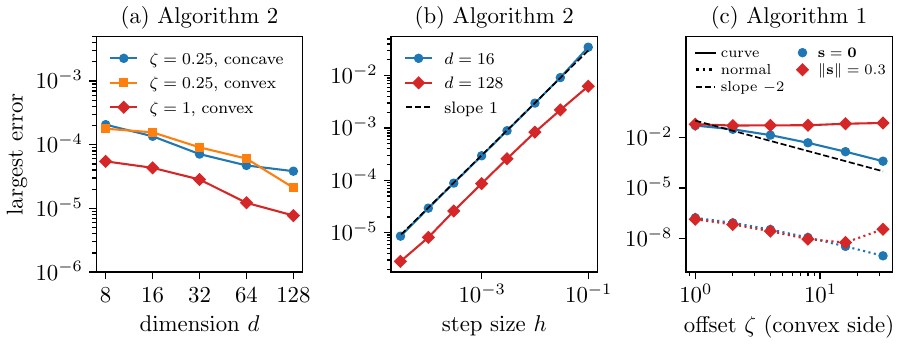}
  \caption{The two algorithms on the graph of \eqref{eqn:curv:test_function}, with exact closest point projections. (a) \Cref{alg:curv:recon_direction}: largest error over the $d-1$ principal curvatures at five random base points with $\norm{\v s} = 0.3$, step size $h = 10^{-4}$, offset $\zeta = 0.25$ on the concave side (all principal curvatures positive with respect to the normal toward $\v x_0$) and $\zeta = 0.25$, $1$ on the convex side. (b) The same error against the step size $h$ for $d = 16$ and $128$, $\zeta = 1$ on the convex side. (c) \Cref{alg:curv:recon_2nd} with $d = 8$, $M = 276$ rays and central differences with $h = 10^{-3}$, against the offset on the convex side, with the curvature of the induced curve (solid) and with the normal curvature of \cref{rmk:curv:meusnier} (dotted), at the origin, where the third-order terms vanish (blue), and at $\norm{\v s} = 0.3$ (red); dashed: slope $-2$, the rate predicted by \cref{thm:curv:error_alg1}.}
  \label{fig:curv:high_dim}
\end{figure}

\section{Submanifolds of higher codimension}\label{sec:curv:codim}

In this section, $\cM\subset\bbR^d$ is a compact $C^2$ submanifold of dimension $d_{\mathrm{in}}$ and codimension $d_{\mathrm{co}} = d-d_{\mathrm{in}}\ge1$ with reach $\iota>0$, and $\cP$ is its closest point projection. Compared with hypersurfaces, two things change. First, the curvature of $\cM$ is described by the shape operators $S_{\bs\nu}$ along the normal vectors $\bs\nu$. The shape operators are linear in $\bs\nu$, so the $d_{\mathrm{co}}$ shape operators for a basis of the $d_{\mathrm{co}}$-dimensional normal space $\cT^\perp_{\v y}\cM$ determine all of them (\cref{sec:curv:prelim_codim}), while the offset $\v x_0-\v y_0$ of a probing point gives only one normal direction, a combination of the basis vectors. Second, a random direction can be mostly normal: at a point $\v y\in\cM$, the part in the normal space $\cT^\perp_{\v y}\cM$ of a random ray with direction $\v r$ uniform on $\bbS^{d-1}$ has squared length $d_{\mathrm{co}}/d$ on average, so the slopes of \cref{sec:curv:concentration} are not small for large $d_{\mathrm{co}}$. M\"arz and Macdonald~\cite{marz2012calculus} extended the calculus of the closest point method to surfaces of any codimension using the structure of the closest point projection near the surface, and they construct the closest point functions for surfaces of higher codimension one codimension at a time. We follow their ideas: we first describe the projection near $\cM$, and then we treat the normal directions one at a time.

\subsection{The closest point projection near a submanifold}\label{sec:curv:codim_projection}
The properties of \cref{sec:curv:cp_property} hold in any codimension, with the same proofs: in \cref{thm:curv:cp_normal}, the normal line becomes the normal space, and the level-set function $F$ becomes a map $F:\bbR^d\to\bbR^{d_{\mathrm{co}}}$ whose $d_{\mathrm{co}}$ components define $\cM$ near $\cP\v x$, with a multiplier $\bs\lambda\in\bbR^{d_{\mathrm{co}}}$ in the Lagrangian. The next proposition extends the derivative \eqref{eqn:curv:dP_hyper} of $\cP$ to codimension $d_{\mathrm{co}}$. We prove it in the coordinates that M\"arz and Macdonald use to show that the Euclidean closest point function satisfies their principles~\cite[proof of Thm.~3.8]{marz2012calculus}; the formula for $\cD\cP$ in (iii) is also Theorem~C of~\cite{leobacher2021existence}.

\begin{proposition}[Closest point projection in codimension $d_{\mathrm{co}}$]\label{thm:curv:dP_codim}
Let $\v y\in\cM$, let $\bs\nu\in\cT^\perp_{\v y}\cM$ be any normal vector with $\norm{\bs\nu}<\iota$, and let $\v x = \v y+\bs\nu$. Then:
\begin{enumerate}[label=(\roman*)]
  \item $\cP\v x = \v y$. Since $\bs\nu$ is arbitrary, the points of the open $\iota$-tubular neighborhood of $\cM$ with closest point being $\v y$ form the normal disk $\v y+\{\bs\nu\in\cT^\perp_{\v y}\cM:\norm{\bs\nu}<\iota\}$;
  \item the eigenvalues of $S_{\bs\nu}$ lie in $[-\norm{\bs\nu}/\iota,\,\norm{\bs\nu}/\iota]$;
  \item $\cP$ is differentiable on the whole open tubular neighborhood, and at $\v x = \v y+\bs\nu$
  \begin{equation}\label{eqn:curv:dP_codim}
    \cD\cP(\v x) = (I-S_{\bs\nu})^{-1}\Proj_{\cT_{\v y}\cM},
  \end{equation}
  where the inverse is taken on $\cT_{\v y}\cM$. In particular, $\cD\cP(\v x)$ is symmetric and vanishes on $\cT^\perp_{\v y}\cM$. When $\bs\nu = \v 0$, $\cD\cP(\v y) = \Proj_{\cT_{\v y}\cM}$.
\end{enumerate}
\end{proposition}
\begin{proof}
(i) If $\bs\nu$ is normal at $\v y$ and $\norm{\bs\nu}<\iota$, then $\cP(\v y+\bs\nu) = \v y$~\cite[Thm.~4.8(12)]{federer1959curvature}. Conversely, $\v x-\cP\v x$ is normal at $\cP\v x$ by \cref{thm:curv:cp_normal}.

(ii) Let $\v v\in\cT_{\v y}\cM$ be a unit vector and $\alpha$ a unit-speed curve on $\cM$ with $\alpha(0) = \v y$ and $\alpha'(0) = \v v$. As in the proof of \cref{thm:curv:quad_growth},
\[\inpd{\bs\nu}{\alpha(s)-\v y}\le\frac{\norm{\bs\nu}\norm{\alpha(s)-\v y}^2}{2\iota}\le\frac{\norm{\bs\nu}s^2}{2\iota}.\]
 Since $\alpha(s)-\v y = s\v v+\frac{s^2}{2}\alpha''(0)+o(s^2)$, $\inpd{\bs\nu}{\v v} = 0$, and $\inpd{\bs\nu}{\alpha''(0)} = \v v^TS_{\bs\nu}\v v$ by \eqref{eqn:curv:meusnier}, the left-hand side is $\frac{s^2}{2}\v v^TS_{\bs\nu}\v v+o(s^2)$. Dividing by $s^2/2$ and letting $s\to0$ gives $\v v^TS_{\bs\nu}\v v\le\norm{\bs\nu}/\iota$. The same argument with $-\bs\nu$ gives the lower bound.

(iii) Let $\chi:U\subset\bbR^{d_{\mathrm{in}}}\to\cM$ be a local $C^2$ parametrization with $\chi(\bs\theta_0) = \v y$, and let $\bs\nu_1,\dots,\bs\nu_{d_{\mathrm{co}}}$ be a $C^1$ orthonormal frame of the normal spaces near $\v y$. As in~\cite{marz2012calculus}, set
\[
  \Phi(\bs\theta,\v t)\defeq\chi(\bs\theta)+\sum_{j=1}^{d_{\mathrm{co}}}t_j\bs\nu_j\big(\chi(\bs\theta)\big),
\]
with $\bs\theta$ being the local coordinates on $\cM$ and $\v t$ being the coordinates in the normal space, and choose $\v t_0$ with $\bs\nu = \sum_jt_{0,j}\bs\nu_j(\v y)$, so that $\Phi(\bs\theta_0,\v t_0) = \v x$. By (i), $\cP(\Phi(\bs\theta,\v t)) = \chi(\bs\theta)$ whenever $\norm{\v t}<\iota$. The derivative $\cD\Phi(\bs\theta_0,\v t_0)$ of $\Phi$ at $(\bs\theta_0,\v t_0)$ maps $\partial/\partial t_j$ to $\bs\nu_j(\v y)$, and $\partial/\partial\theta_l$ to $\v v_l+\sum_jt_{0,j}\cD_{\v v_l}\bs\nu_j$, where $\v v_l = \partial_l\chi(\bs\theta_0)$.

By the definition of the shape operator, the tangential part of $\sum_jt_{0,j}\cD_{\v v_l}\bs\nu_j$ is $-S_{\bs\nu}\v v_l$. So, writing $\partial_{\theta_l}\Phi = (I-S_{\bs\nu})\v v_l+\v z_l$ with $\v z_l$ normal, and since the columns $\partial_{t_j}\Phi = \bs\nu_j(\v y)$ have no tangential part, in the decomposition $\bbR^d = \cT_{\v y}\cM\oplus\cT^\perp_{\v y}\cM$
\[
  \cD\Phi(\bs\theta_0,\v t_0) = \begin{bmatrix}(I-S_{\bs\nu})V & 0\\ Z & N\end{bmatrix},
\]
with $V = [\v v_1,\dots,\v v_{d_{\mathrm{in}}}]$, $Z = [\v z_1,\dots,\v z_{d_{\mathrm{in}}}]$ and $N = [\bs\nu_1(\v y),\dots,\bs\nu_{d_{\mathrm{co}}}(\v y)]$. This matrix is block triangular with the diagonal blocks $(I-S_{\bs\nu})V$ and $N$, both invertible by (ii). Hence, by the inverse function theorem, $\Phi$ is a $C^1$ diffeomorphism near $(\bs\theta_0,\v t_0)$, and near $\v x$ we have $\cP = \chi\circ\pi\circ\Phi^{-1}$ with $\pi(\bs\theta,\v t) = \bs\theta$, so $\cP$ is differentiable at $\v x$. By the chain rule, differentiating $\cP\circ\Phi = \chi\circ\pi$ at $(\bs\theta_0,\v t_0)$ gives
\[
  \cD\cP(\v x)\,\partial_{t_j}\Phi = \partial_{t_j}(\chi\circ\pi) = \v 0,\qquad \cD\cP(\v x)\,\partial_{\theta_l}\Phi = \partial_{\theta_l}\chi = \v v_l,
\]
that is, $\cD\cP(\v x)\bs\nu_j(\v y) = \v 0$ and $\cD\cP(\v x)\big((I-S_{\bs\nu})\v v_l+\v z_l\big) = \v v_l$, where $\cD\cP(\v x)\v z_l = \v 0$ because $\v z_l$ is normal. Thus $\cD\cP(\v x)(I-S_{\bs\nu})\v v = \v v$ for all $\v v\in\cT_{\v y}\cM$, and $\cD\cP(\v x)$ vanishes on the normal space. Since $S_{\bs\nu}$ is self-adjoint, so is $\cD\cP(\v x)$.
\end{proof}

For a hypersurface, $\bs\nu = \zeta\v n$ and $S_{\bs\nu} = \zeta S$, and (iii) is \eqref{eqn:curv:dP_hyper}. By (i), the preimages of the Euclidean projection $\cP$ are flat normal disks. We use this structure in two ways: to see what a probing point measures, and to move a probing point inside the preimage of a fixed closest point.

\subsection{Moving along a ray}\label{sec:curv:codim_probe}
The first-order formula of \cref{thm:curv:dP_codim} needs only $C^2$. To bound the remainders by $\cO(h^2)$, we let $\cM$ be of class $C^3$ from now on, so that $\cP$ is $C^2$ within the reach~\cite{marz2012calculus,leobacher2021existence}; for a $C^2$ submanifold, the remainders are only $o(h)$. For a probing point $\v x_0 = \v y_0+\zeta\v n$, where $\v n$ is a unit normal at $\v y_0$ and $\zeta<\iota$, and a ray in the direction $\v r$, \cref{thm:curv:dP_codim} gives
\begin{equation}\label{eqn:curv:codim_ray}
  \cP(\v x_0+h\v r) = \v y_0+h\,(I-\zeta S_{\v n})^{-1}\Proj_{\cT_{\v y_0}\cM}\v r+\cO(h^2).
\end{equation}
The displacement depends on the curvature of $\cM$ only through $S_{\v n}$, the shape operator in the direction of the offset. The shape operators in the remaining normal directions orthogonal to $\v n$ do not affect the first order. Write $\v r = \v r_T+\v g$ with $\v r_T$ tangent and $\v g$ normal at $\v y_0$, so that one step of size $h$ along the ray moves the probing point by $h\v r_T$ in the tangent directions and by $h\v g$ in the normal directions at the same time, as $\ell(h)$ in \eqref{eqn:curv:l_h} does for a hypersurface. Moving along the normal only takes the probing point inside the normal disk of $\v y_0$, from $\v y_0+\zeta\v n$ to $\v y_0+\zeta\v n+h\v g$, which does not move the closest point. The closest point moves only with the tangential move $h\v r_T$: applying \cref{thm:curv:dP_codim} at $\v y_0+\zeta\v n+h\v g$ gives
\begin{equation}\label{eqn:curv:codim_ray_normal}
  \cP(\v x_0+h\v r) = \v y_0+h\,(I-\zeta S_{\v n}-hS_{\v g})^{-1}\v r_T+\cO(h^2) = \v y_0+h\,(I-\zeta S_{\v n})^{-1}\v r_T+\cO(h^2).
\end{equation}
So the normal part of the ray changes the offset vector by $h\v g$, and with it the shape operator by $hS_{\v g}$, both of the order of the step size $h$, which affects the displacement only at second order. For a hypersurface, $\v g = \gamma\v n$, and this is the change of the offset from $\zeta$ to $c(h) = \zeta+\gamma h+\cO(h^2)$ in \cref{thm:curv:parametrization}. Therefore the large normal parts of random rays in high codimension do not affect the first-order displacement.

\subsection{One normal direction at a time}\label{sec:curv:codim_alg}
By \eqref{eqn:curv:codim_ray}, one probing point gives the shape operator along its offset $\v n$. We take $\bs\nu_1 = \v n$ and complete it to an orthonormal basis $\bs\nu_1,\dots,\bs\nu_{d_{\mathrm{co}}}$ of the normal space; as noted at the start of this section, $S_{\bs\nu_1},\dots,S_{\bs\nu_{d_{\mathrm{co}}}}$ determine the second fundamental form. By \cref{thm:curv:dP_codim}(i), the probing points $\v y_0+\zeta\bs\nu_j$, $j = 1,\dots,d_{\mathrm{co}}$, all have the closest point $\v y_0$. So we can treat the normal directions one at a time, in the same way as M\"arz and Macdonald deal with a surface of higher codimension one codimension at a time~\cite[Sec.~5.2]{marz2012calculus} (see also \cref{rmk:curv:flag}). The first probing point is the sampled point $\v x_0$, and its rays also estimate the tangent space: by \eqref{eqn:curv:codim_ray}, the matrix $WU^T$ of \cref{sec:curv:recon_direction} has rank $d_{\mathrm{in}}$ up to $\cO(h^2)$, and its range is $\cT_{\v y_0}\cM$. The other probing points are placed on the normal disk of $\v y_0$, and their rays are taken along the tangent space. \Cref{alg:curv:codim} summarizes the procedure.

\begin{algorithm}[htbp]
\caption{Second fundamental form of a submanifold of codimension $d_{\mathrm{co}}$}\label{alg:curv:codim}
\begin{algorithmic}
  \STATE \textbf{Input:} $\cM$: a compact $C^2$ submanifold of dimension $d_{\mathrm{in}}$ in $\bbR^d$ with reach $\iota>0$; $h$: the step size. All probing points $\v x_0,\v x_j$ lie within distance $\iota$ of $\cM$.
  \STATE
  \STATE Sample $\v x_0$; find $\v y_0 = \cP\v x_0$; record $\zeta = \norm{\v x_0-\v y_0}$ and $\bs\nu_1 = (\v x_0-\v y_0)/\zeta$.
  \STATE Generate an orthonormal basis $\{\v u_i\}_{i=1}^d$ of $\bbR^d$.
  \STATE Compute $\v w_i = \cP(\v x_0+h\v u_i)-\v y_0$.
  \STATE Eigen-decompose $B = \frac12(WU^T+UW^T)$ and obtain the $d_{\mathrm{in}}$ eigenpairs $\{(\mu_i, \v q_i)\}_{i=1}^{d_{\mathrm{in}}}$ with the eigenvalues of largest modulus.
  \STATE Set $Q = [\v q_1,\dots,\v q_{d_{\mathrm{in}}}]$; its columns form a basis of the estimated tangent space $\hat T$.
  \STATE $\hat S_1 := \diag\big((1-h/\mu_i)/\zeta\big)$ collects the estimated principal curvatures along $\bs\nu_1$.
  \STATE Complete $\bs\nu_1$ to an orthonormal basis $\bs\nu_1,\dots,\bs\nu_{d_{\mathrm{co}}}$ of $\hat T^\perp$.
  \FOR{$j = 2:d_{\mathrm{co}}$}
    \STATE Set $\v x_j = \v y_0+\zeta\bs\nu_j$; find $\v y_j = \cP\v x_j$ and $\zeta_j = \norm{\v x_j-\v y_j}$.
    \STATE Compute $W_j = [\cP(\v x_j+h\v q_i)-\v y_j]_{i=1}^{d_{\mathrm{in}}}$, $B_j = \frac12(Q^TW_j+W_j^TQ)$ and $\hat S_j = (I-hB_j^{-1})/\zeta_j$.
  \ENDFOR
  \STATE \textbf{Output:} $\hat\sff(\v v,\v w) = \sum_{j=1}^{d_{\mathrm{co}}}\inpd{\hat S_jQ^T\v v}{Q^T\v w}\bs\nu_j$ for $\v v,\v w\in\hat T$.
\end{algorithmic}
\end{algorithm}

\begin{theorem}\label{thm:curv:codim}
Let $\cM$ be of class $C^3$ (as in \cref{sec:curv:codim_probe}; the algorithm itself only needs $C^2$), let $\v x_0 = \v y_0+\zeta\bs\nu_1$ with $0<\zeta<\iota$, and let $\kappa_1(\bs\nu),\dots,\kappa_{d_{\mathrm{in}}}(\bs\nu)$ be the eigenvalues of $S_{\bs\nu}$. Then, as $h\to0$, the quantities of \cref{alg:curv:codim} satisfy:
\begin{enumerate}[label=(\roman*)]
  \item $B = h(I-\zeta S_{\bs\nu_1})^{-1}\Proj_{\cT_{\v y_0}\cM}+\cO(h^2)$; its $d_{\mathrm{in}}$ eigenvalues of largest modulus are $\mu_i = h/(1-\zeta\kappa_i(\bs\nu_1))+\cO(h^2)$, and the other $d_{\mathrm{co}}$ eigenvalues are $\cO(h^2)$;
  \item the estimated principal curvatures along $\bs\nu_1$ satisfy
  \begin{equation}\label{eqn:curv:codim_rate_first}
    \frac{1-h/\mu_i}{\zeta} = \kappa_i(\bs\nu_1)+\cO(h/\zeta),
  \end{equation}
  and $\norm{QQ^T-\Proj_{\cT_{\v y_0}\cM}} = \cO(h)$;
  \item for $j\ge2$, $\norm{\v y_j-\v y_0} = \cO(\zeta h)$ and
  \begin{equation}\label{eqn:curv:codim_rate_other}
    \hat S_j = Q^TS_{\bs\nu_j}Q+\cO(h/\zeta)+\cO(h).
  \end{equation}
\end{enumerate}
The constants in the $\cO$-terms depend on $\iota-\zeta$ and on bounds for the first two derivatives of $\cP$ near the probing points.
\end{theorem}

\begin{proof}
\begin{enumerate}[label=(\roman*)]
  \item Since $\cP$ is $C^2$, $\v w_i = h\cD\cP(\v x_0)\v u_i+\cO(h^2)$, and since $U$ is orthogonal, $WU^T = h\cD\cP(\v x_0)+\cO(h^2)$. By \cref{thm:curv:dP_codim}, $\cD\cP(\v x_0)$ is symmetric, with the eigenvalues
  \[\frac{1}{1-\zeta\kappa_i(\bs\nu_1)}\in\Big[\frac{1}{1+\zeta/\iota},\,\frac{1}{1-\zeta/\iota}\Big]\]
  on the tangent space and $0$ on the normal space. So $B$ differs from $h\cD\cP(\v x_0)$ by $\cO(h^2)$, and Weyl's inequality gives the eigenvalues.
  \item By (i), $h/\mu_i = (1-\zeta\kappa_i)(1+\cO(h))$, and $\abs{\zeta\kappa_i}<1$ gives the first claim. The eigenvalues of $h\cD\cP(\v x_0)$ on the tangent space are separated from $0$ by at least $h/(1+\zeta/\iota)$, so by the Davis--Kahan theorem~\cite{davis1970rotation} the distance from the span of the corresponding eigenvectors of $B$ to $\cT_{\v y_0}\cM$ is
  \[\cO(h^2)/(h/(1+\zeta/\iota)) = \cO(h).\]
  \item In \cref{alg:curv:codim}, $\bs\nu_j$ is a unit vector orthogonal to $\hat T$, and $\hat T$ is within $\cO(h)$ of $\cT_{\v y_0}\cM$ by (ii). So $\bs\nu_j = \bs\nu_j'+\cO(h)$, where $\bs\nu_j'$ is the normalized projection of $\bs\nu_j$ onto the normal space $\cT^\perp_{\v y_0}\cM$. By \cref{thm:curv:dP_codim}(i), $\cP(\v y_0+\zeta\bs\nu_j') = \v y_0$, so $\v y_j-\v y_0 = \cO(\zeta h)$ because $\cP$ is Lipschitz near $\v x_j$. At $\v y_j$, the offset
  \[\v x_j-\v y_j= \zeta\bs\nu_j-(\v y_j-\v y_0)\] is a normal vector of length $\zeta_j = \zeta+\cO(\zeta h)$ whose direction differs from $\bs\nu_j$, and hence from $\bs\nu_j'$, by $\cO(h)$, and the tangent spaces at $\v y_j$ and $\v y_0$ differ by $\cO(\zeta h)$. Applying \eqref{eqn:curv:codim_ray} at $\v x_j$ to the rays $\v q_i$, which lie in the tangent space up to $\cO(h)$, gives $B_j = h(I-\zeta_jQ^TS_{\bs\nu_j}Q)^{-1}+\cO(h^2)+\cO(\zeta h^2)$, and inverting as in (ii) gives the claim.
\end{enumerate}
\end{proof}

\Cref{alg:curv:codim} uses $1+d+(d_{\mathrm{co}}-1)(d_{\mathrm{in}}+1)$ closest point projections, one eigendecomposition of a $d\times d$ matrix and $d_{\mathrm{co}}-1$ inversions of $d_{\mathrm{in}}\times d_{\mathrm{in}}$ matrices. For a hypersurface ($d_{\mathrm{co}} = 1$), it is \cref{alg:curv:recon_direction} with symmetric $B$. For $d_{\mathrm{in}}\ll d$, it needs about $(d_{\mathrm{in}}+2)d$ projections, which grows linearly in $d$ for a manifold of fixed dimension. The $d_{\mathrm{in}}$ eigenvalues of order $h$ are separated from the $d_{\mathrm{co}}$ eigenvalues of order $h^2$, so the gap also shows the dimension $d_{\mathrm{in}}$ of $\cM$ when it is not known. Vanishing shape operators, $S_{\bs\nu} = 0$, also show the flat normal directions of $\cM$; these are the directions that make local linear regression on $\cM$ ill-posed~\cite{liu2023linear}.

\begin{remark}[$\cM$ as a nested hypersurface]\label{rmk:curv:flag}
Let $\cM = S_1\cap\dots\cap S_{d_{\mathrm{co}}}$ be a proper intersection of hypersurfaces, as in the construction of~\cite[Sec.~5.2]{marz2012calculus}, and let $\cM_j = S_1\cap\dots\cap S_j$, so that $\cM = \cM_{d_{\mathrm{co}}}\subset\cM_{d_{\mathrm{co}}-1}\subset\dots\subset\cM_0 = \bbR^d$ and each $\cM_j$ is a hypersurface of $\cM_{j-1}$.

Let $\bs\nu_j$ be the unit normal of $\cM_j$ in $\cM_{j-1}$ at $\v y\in\cM$. The normal is taken inside $\cM_{j-1}$, so $\bs\nu_j\in\cT_{\v y}\cM_{j-1}$, and it is orthogonal to $\cT_{\v y}\cM_j$, which has codimension one in $\cT_{\v y}\cM_{j-1}$; this determines $\bs\nu_j$ up to sign. For $i<j$, $\bs\nu_i$ is orthogonal to $\cT_{\v y}\cM_i$, which contains $\cT_{\v y}\cM_{j-1}$, and hence to $\bs\nu_j$. So $\bs\nu_1,\dots,\bs\nu_{d_{\mathrm{co}}}$ is an orthonormal basis of $\cT^\perp_{\v y}\cM$. For a curve $\alpha$ on $\cM$ with $\alpha(0) = \v y$, differentiating $\inpd{\alpha'}{\tilde{\bs\nu}_j\circ\alpha} = 0$ as in \cref{sec:curv:prelim_codim} gives $\inpd{\sff_{\v y}(\v v,\v v)}{\bs\nu_j} = \sff^{(j)}_{\v y}(\v v,\v v)$ with $\v v = \alpha'(0)$, where $\sff^{(j)}$ is the second fundamental form of the hypersurface $\cM_j$ in $\cM_{j-1}$. Hence
\begin{equation}\label{eqn:curv:sff_nested}
  \sff_{\v y}(\v v,\v w) = \sum_{j=1}^{d_{\mathrm{co}}}\sff^{(j)}_{\v y}(\v v,\v w)\,\bs\nu_j,\qquad\v v,\v w\in\cT_{\v y}\cM.
\end{equation}
So $S_{\bs\nu_j}$ is the shape operator of $\cM_j$ in $\cM_{j-1}$, restricted to $\cT_{\v y}\cM$, and with these normal directions each step of \cref{alg:curv:codim} is a hypersurface problem, in the same order as the closest point function of M\"arz and Macdonald, which projects onto $S_1$ first, then within $S_1$ onto $\cM_2$, and so on.
\end{remark}

\begin{remark}[Closest point functions of M\"arz and Macdonald~\cite{marz2012calculus}]\label{rmk:curv:cpf}
For $\cM = \{\varphi_1 = \dots = \varphi_{d_{\mathrm{co}}} = 0\}$, a proper intersection of level sets, $\cP\v x$ has in general no closed form, and each query needs a constrained minimization. M\"arz and Macdonald~\cite[Sec.~5]{marz2012calculus} instead retract onto the level sets one at a time, in the order of \cref{rmk:curv:flag}, and obtain a \emph{closest point function} $\mathrm{cp}$: a $C^1$ map from a neighborhood of $\cM$ onto $\cM$ that is the identity on $\cM$ and satisfies $\cD\mathrm{cp}(\v y) = \Proj_{\cT_{\v y}\cM}$ for $\v y\in\cM$~\cite[Def.~3.1]{marz2012calculus}. Such a function can replace $\cP$ in \cref{alg:curv:codim}, but the estimate is less accurate. The preimages of $\mathrm{cp}$ meet $\cM$ orthogonally~\cite[Thm.~3.7]{marz2012calculus}, but unlike the flat normal disks of $\cP$ (\cref{thm:curv:dP_codim}(i)), they bend by $\cO(\zeta^2)$ in the tangent directions at distance $\zeta$ from $\cM$. This changes the derivative at $\v x_0$ by $\cO(\zeta^2)$ compared with \eqref{eqn:curv:dP_codim}, and since the curvature is computed from the term $\zeta S_{\v n}$ of this derivative, the curvature error is of order $\zeta$, in addition to $\cO(h/\zeta)$. So the probing points must stay close to $\cM$, and the best error, at $\zeta\sim\sqrt h$, is of order $\sqrt h$, while with $\cP$ the error is $\cO(h/\zeta)+\cO(h)$ at every offset within the reach (\cref{thm:curv:codim}).
\end{remark}

\begin{remark}[The whole normal curvature vector from one probing point]\label{rmk:curv:codim_meusnier}
Meusnier's theorem \eqref{eqn:curv:meusnier} gives all $d_{\mathrm{co}}$ normal directions at once, at second order. For every ray from $\v x_0$, the normal part of the acceleration of the induced curve is $\sff_{\v y_0}(\v t,\v t)$, where $\v t$ is its velocity. With the tangent space $\hat T$ of \cref{alg:curv:codim}, the central second difference
\begin{equation}\label{eqn:curv:second_difference} \frac{\Proj_{\hat T^\perp}\big(\cP(\v x_0+h\v r)-2\v y_0+\cP(\v x_0-h\v r)\big)}{h^2}\end{equation} approximates this vector, and a least squares problem as in \eqref{eqn:curv:lstsq}, with $d_{\mathrm{co}}$ right-hand sides, over $M\ge d_{\mathrm{in}}(d_{\mathrm{in}}+1)/2$ rays recovers $\sff_{\v y_0}$ from one probing point. For the graph surface of \cref{fig:curv:codim} in $\bbR^6$, $\bbR^{30}$ and $\bbR^{100}$, with $M = 12$ rays and central differences also for the tangent space, the relative error was below $2\times10^{-4}$ for $h = 10^{-2}$ and below $2\times10^{-6}$ for $h = 10^{-3}$. As in \cref{rmk:curv:meusnier}, the price is the use of the normal space, here the estimated $\hat T^\perp$.
\end{remark}

\subsection{Experiments}\label{sec:curv:codim_experiments}
We test \cref{alg:curv:codim} on two surfaces (intrinsic dimension $d_{\mathrm{in}} = 2$) with exact closest point projections, each rotated into $\bbR^d$ at random: the flat torus $\bbS^1(1)\times\bbS^1(0.5)\subset\bbR^4$, and the graph surface $\{(\v s,\v f(\v s))\}$ with $f_j(\v s) = \frac12\v s^TA_j\v s+\frac14\norm{\v s}^4b_j$, $j = 1,\dots,4$, with random symmetric $A_j$ and random $b_j$. For larger $d$, the additional normal directions are flat. The error is $\norm{\hat\sff-\sff}/\norm{\sff}$ with $\norm{\sff}^2 = \sum_j\norm{S_{\bs\nu_j}}_F^2$, at five random base points with offset $\zeta = 0.3$. It decreases linearly in $h$, as \cref{thm:curv:codim} predicts, both in $\bbR^6$ and in $\bbR^{100}$ (\cref{fig:curv:codim}(a)). At $h = 10^{-3}$, it stays between $5\times10^{-5}$ and $1.4\times10^{-3}$ for every codimension from $2$ to $98$ (\cref{fig:curv:codim}(b)), at a cost of $1+d+(d_{\mathrm{co}}-1)(d_{\mathrm{in}}+1)$ projections, $392$ for $d_{\mathrm{co}} = 98$.

\begin{figure}[htbp]
  \centering
  \includegraphics[width=0.9\textwidth]{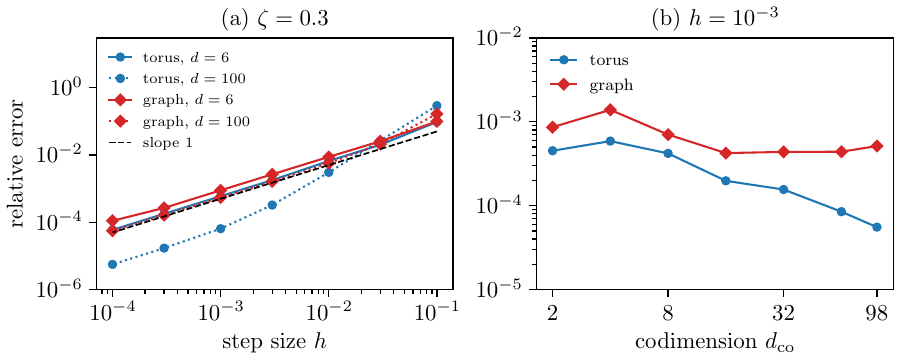}
  \caption{\Cref{alg:curv:codim} on surfaces of dimension $d_{\mathrm{in}} = 2$ with exact closest point projections. (a) Relative error of the second fundamental form against the step size $h$, for the flat torus and the graph surface in $\bbR^6$ and $\bbR^{100}$, offset $\zeta = 0.3$, median over five base points. (b) The same error against the codimension $d_{\mathrm{co}} = d-2$, for offset $\zeta = 0.3$ and step size $h = 10^{-3}$.}
  \label{fig:curv:codim}
\end{figure}

\section{Computations of curvatures on point clouds}\label{sec:curv:point_cloud}

In this section, $\cM$ is again a hypersurface. While our earlier analysis and algorithms are formulated under smoothness and exact projection assumptions, practical applications typically provide only finite point samples of the underlying surface. Since~\cref{alg:curv:recon_direction} delivers controllable error bounds even in the absence of tangent information, it is natural to ask how it performs when the exact closest point projection is replaced by a nearest-neighbor query on a finite point cloud.

\subsection{Error bounds of~\texorpdfstring{\cref{alg:curv:recon_direction}}{the direct algorithm} on point clouds}\label{sec:curv:error_point_cloud}
\begin{definition}[$\varepsilon$-net/$\varepsilon$-sample]\label{def:curv:eps_net}
Given $\varepsilon>0$, a set $\Gamma$ is an $\varepsilon$-net of the manifold $(\cM, d_{\cM})$, with the geodesic distance $d_\cM$, if $\Gamma \subset \cM$ and every $\v y\in\cM$ has a point $\v y'\in\Gamma$ with $d_\cM(\v y',\v y)\le\varepsilon$. Equivalently, a point cloud $\Gamma\subset\cM$ is an $\varepsilon$-sample of $\cM$ if the Hausdorff distance with respect to $d_\cM$ satisfies
\begin{equation}\label{eqn:curv:eps_sample}
d_H(\Gamma, \cM) = \sup_{\v y \in \cM} \inf_{\v y' \in \Gamma} d_\cM(\v y',\v y) \le \varepsilon .
\end{equation}
\end{definition}
\begin{remark}
The Euclidean distance between two points of $\cM$ is at most their geodesic distance, $\norm{\v y'-\v y}\le d_\cM(\v y',\v y)$, so an $\varepsilon$-net of $(\cM,d_\cM)$ is also an $\varepsilon$-net for the Euclidean distance. Therefore, for convenience of analysis, we adopt the Euclidean metric when defining $\varepsilon$-nets in what follows.
\end{remark}
Let $\cP_\Gamma : \bbR^d \to \Gamma$ denote the nearest neighbor projection onto the point cloud, as in~\cite{liu2024nearest}, with $\cP_\Gamma\v x = \tilde{\v y} \in \Gamma$. It is constant on each Voronoi cell of $\Gamma$, so along a ray it changes only where the ray crosses into another cell; this is the source of the discretization effects below.
We now investigate the deviation between the exact projection $\v y= \cP\v x$ and its discrete counterpart $\tilde{\v y}$.

Consider again the aligned local frame (Frame--1) introduced in~\cref{sec:curv:recon_direction}, now at $\v y = \cP\v x$, so that $\v y$ is the origin and $\v x-\v y$ is along $\v e_d$, where locally the manifold $\cM$ is represented as the graph of a function $f$ with $f(\v 0)=0$, $\nabla f(\v 0)=\v 0$. Assume $\norm{\Hess f} \le L$ on the ball of radius $\varepsilon$ around $\v 0$, with the same $L$ at every closest point used below. We also use the quadratic growth of \cref{thm:curv:quad_growth}.
\begin{lemma}[Bounds on deviation of the projections]\label{thm:curv:bound_proj_err}
Fix $\v x\notin\cM$ within the reach of $\cM$, with $\zeta = \norm{\v x-\cP\v x}$, and let $\Gamma$ be an $\varepsilon$-sample of $\cM$. Then $\v y$ and $\tilde{\v y}$, the projections of $\v x$ onto $\cM$ and $\Gamma$ respectively, satisfy
\begin{equation}\label{eqn:curv:proj_deviation}
\norm{\v y - \tilde{\v y}} \le \varepsilon \sqrt{\frac{1 + \zeta L + \frac{1}{4} L^2 \varepsilon^2}{1-\zeta/\iota}} =:\delta(\v x).
\end{equation}
\end{lemma}
\begin{proof}
Since $\Gamma$ is an $\varepsilon$-sample, there exists $\v y' \in \Gamma$ such that $\norm{\v y' - \v y} \le \varepsilon$. In local coordinates, $\v y' = (\v s, f(\v s))$ with $\norm{\v s} \le \varepsilon$ and, by Taylor's theorem, $\abs{f(\v s)} \le \tfrac{1}{2} L \varepsilon^2$. Since $\v x-\v y$ is normal to $\cT_{\v y}\cM$,
\[
\norm{\v y' - \v x}^2 = \norm{\v s}^2 + \big(\zeta - f(\v s)\big)^2 \le \varepsilon^2 + \big(\zeta + \tfrac{1}{2} L \varepsilon^2\big)^2 = \zeta^2+\varepsilon^2\big(1+\zeta L+\tfrac14L^2\varepsilon^2\big).
\]
By definition of $\tilde{\v y}$, $\norm{\tilde{\v y} - \v x} \le \norm{\v y' - \v x}$. On the other hand, \cref{thm:curv:quad_growth} with $\v z = \tilde{\v y}\in\Gamma\subset\cM$ gives $\norm{\tilde{\v y}-\v x}^2\ge\zeta^2+(1-\zeta/\iota)\norm{\tilde{\v y}-\v y}^2$. Combining the two inequalities and subtracting $\zeta^2$ yields the claim.
\end{proof}
The preliminary estimates on the deviation of $\tilde{\v y}=\cP_\Gamma\v x$ provide tools to quantify the errors when~\cref{alg:curv:recon_direction} is applied to $\Gamma$, where the matrix $WU^T$ in~\eqref{eqn:curv:key_alg2} is replaced by $\tilde W U^T$ with columns of $\tilde W$ being $\tilde{\v w}_i = \cP_\Gamma(\v x + h \v u_i) - \cP_\Gamma\v x$, denoting the displacement of the projection onto $\Gamma$. Since $\tilde WU^T$ is not symmetric, we take the eigenvalues of its symmetric part $\frac12(\tilde WU^T+U\tilde W^T)$, and compare them with those of the symmetric part of $WU^T = h\cD\cP(\v x)+\cO(h^2)$. Then the primary objective is to control
$\norm{WU^T - \tilde W U^T}_2 = \norm{W - \tilde W}_2$ (since $U$ is orthogonal), which ultimately allows us to bound the deviation of the curvature estimates.
\begin{lemma}[Bounds on deviation of the eigenvalues]\label{thm:curv:bound_eval_ptcloud}
Let $\delta$ be the maximum of $\delta(\v x_0)$ and $\delta(\v x_0+h\v u_i)$, $i=1,\dots,d$, from \cref{thm:curv:bound_proj_err}. Then the eigenvalues $\tilde\mu_1\le\dots\le\tilde\mu_d$ of the symmetric part of $\tilde WU^T$ and $\mu_1\le\dots\le\mu_d$ of the symmetric part of $WU^T$ satisfy
\begin{equation}\label{eqn:curv:eig_perturbation}
\abs{\tilde\mu_i-\mu_i}\le2\sqrt{d}\,\delta,\qquad i=1,\dots,d.
\end{equation}
\end{lemma}
\begin{proof}
A standard bound on the spectral norm yields $\norm{W - \tilde W}_2 \le \sqrt{d} \max_i \norm{\v w_i - \tilde{\v w}_i}$. To bound $\v w - \tilde{\v w}$, observe
\[\v w - \tilde{\v w}
= \cP(\v x + h\v u) - \cP\v x
   - \cP_\Gamma(\v x + h\v u) + \cP_\Gamma\v x
= \cP(\v x + h\v u) - \cP_\Gamma(\v x + h\v u)
   + \tilde{\v y} - \v y ,\]
where both terms are bounded by~\cref{thm:curv:bound_proj_err}, the first one at the base point $\v x + h\v u$ instead of $\v x$. Hence $\norm{\v w - \tilde{\v w}}\le2\delta$ and $\norm{W-\tilde W}_2\le2\sqrt d\,\delta$. Finally, the symmetric parts of $\tilde WU^T$ and $WU^T$ differ by the symmetric part of $(\tilde W-W)U^T$, whose norm is at most $\norm{(\tilde W-W)U^T}_2 = \norm{\tilde W-W}_2$, and Weyl's inequality gives the claim.
\end{proof}
Even with an $\cO(\sqrt{d}\delta)$ perturbation bound, it is still necessary to ensure that the estimate is meaningful relative to the scale of the eigenvalues. Recall from the proof of~\cref{thm:curv:error_alg2} that
\begin{equation}\label{eqn:curv:mu_effective}
\mu_i = \frac{h}{1-\kappa_i\zeta_i},\qquad \zeta_i = \zeta+\cO\big(\max_j\abs{\gamma_j}h\big),
\end{equation}
with the effective offsets $\zeta_i$; by the same argument, this holds for the symmetric part of $WU^T$ as well. We now translate the eigenvalue bounds into error bounds on the computed curvatures.

\begin{corollary}[Bounds on deviation of the computed curvatures]\label{thm:curv:curv_ptcloud}
For a fixed principal curvature $\kappa_i$, denote by $\hat\kappa_i$ the corresponding curvature computed by~\cref{alg:curv:recon_direction} on $\cM$, and by $\hat\kappa_i^\Gamma$ the curvature computed on $\Gamma$. Let $\delta_\mu = 2\sqrt d\,\delta$, let $\zeta_i$ be the effective offsets above, and assume $\delta_\mu\abs{1-\kappa_i\zeta_i}<h$. Then
\begin{equation}\label{eqn:curv:ptcloud_bound}
\abs{\hat\kappa_i - \hat\kappa_i^\Gamma}\le
\frac{\big(1-\kappa_i\zeta_i\big)^2\, \delta_\mu}
{\zeta\,\big(h -\abs{1-\kappa_i\zeta_i}\,\delta_\mu\big)},
\end{equation}
which grows as $h\to 0$ or $\zeta\to\infty$, but tends to $0$ as $\varepsilon\to 0$ as expected.
\end{corollary}
\begin{proof}
From~\cref{thm:curv:recon_direction}, $\hat\kappa_i = (1 - h/\mu_i)/\zeta$. On $\Gamma$, we would compute $\hat\kappa_i^\Gamma = (1 - h/\tilde\mu_i)/\zeta$, and
\[
\abs{\hat\kappa_i-\hat\kappa_i^\Gamma}= \abs{\frac{h(\mu_i - \tilde\mu_i)}{\mu_i \tilde\mu_i \zeta}}
\le\frac{h\delta_\mu}{\abs{\mu_i}\big(\abs{\mu_i}-\delta_\mu\big)\zeta},
\]
by \cref{thm:curv:bound_eval_ptcloud}, as long as $\delta_\mu<\abs{\mu_i}$. Substituting $\mu_i = h/(1-\kappa_i\zeta_i)$ gives the stated bound.
\end{proof}
To interpret the bound, assume $L = \cO(1)$ and $h$ is moderately small; for large $\zeta$, $\abs{1-\kappa_i\zeta_i} = \cO(\zeta)$ and $\delta_\mu = \cO(\sqrt{d\zeta}\,\varepsilon)$. Since $\zeta$ exceeds the reach, we use \cref{asm:curv:convex_side}: then $f\le0$ near $\v 0$, so $\norm{\v z-\v x}^2 = \norm{\v s}^2+(\zeta-f(\v s))^2\ge\zeta^2+\norm{\v z-\v y}^2$ for $\v z = (\v s,f(\v s))$ near $\v y$, and the factor $1-\zeta/\iota$ in \cref{thm:curv:bound_proj_err} can be replaced by $1$. Taking either $h \to 0$ or $\zeta \to \infty$, as for the smooth case in~\cref{sec:curv:alg2_error}, violates the condition $\delta_\mu\abs{1-\kappa_i\zeta_i}<h$, and the bound gives no control: the perturbation of the eigenvalues is then as large as the eigenvalues themselves. The only regime in which convergence is recovered is when $\varepsilon \to 0$ while $h$ remains moderately small. In this case, the bound simplifies to
\begin{equation}\label{eqn:curv:ptcloud_rate}
\abs{\hat\kappa_i^\Gamma - \hat\kappa_i} \le\cO\left( \frac{\sqrt{d\zeta}\, \zeta \varepsilon}{h} \right),
\end{equation}
so $h$ and $\zeta$ enter in the opposite way from the error presented in~\cref{thm:curv:error_alg2}. For large $\zeta$ or small $h$, the projected points will remain in a very small neighborhood where discretization effects may dominate and significantly distort the displacement of the projection relative to the smooth setting.

We check the bound on $N$ uniform samples of the unit sphere in $\bbR^3$, for which $\varepsilon = 0.072$, $0.023$ and $0.009$ for $N = 10^4$, $10^5$ and $10^6$, by comparing the curvatures that \cref{alg:curv:recon_direction} computes with the nearest neighbor projection $\cP_\Gamma$ onto the samples and with the exact closest point projection onto the sphere, $\cP\v x = \v x/\norm{\v x}$, for the same base points and rays (\cref{fig:curv:ptcloud_bound}). The error decreases like $\varepsilon/h$, as the bound predicts, until $h/(1+\zeta)$ comes down to the sample spacing; below it, a ray can stay in the Voronoi cell of $\cP_\Gamma\v x_0$, and the estimate fails.

The bound holds in all trials and has the right rate, but it is $30$ to $40$ times larger than the median error, since it takes the worst projection error $2\delta$ for every ray together with the factor $\sqrt d$, and its condition $\delta_\mu\abs{1-\kappa_i\zeta_i}<h$ asks for steps about ten times longer than needed. It also grows faster in $\zeta$ than the error, which grows like $(1+\zeta)^2/\zeta$. The factor $\sqrt{1+\zeta L}$ in $\delta$ is attained when $\cM$ has a flat direction, but on the sphere $\norm{\v z-\v x}^2 = \zeta^2+(1+\zeta)\norm{\v z-\v y}^2$, so the sample nearest to $\v x$ is the one nearest to $\v y$, and $\delta\le\varepsilon$.
\begin{figure}[htbp]
  \centering
  \includegraphics[width=0.9\textwidth]{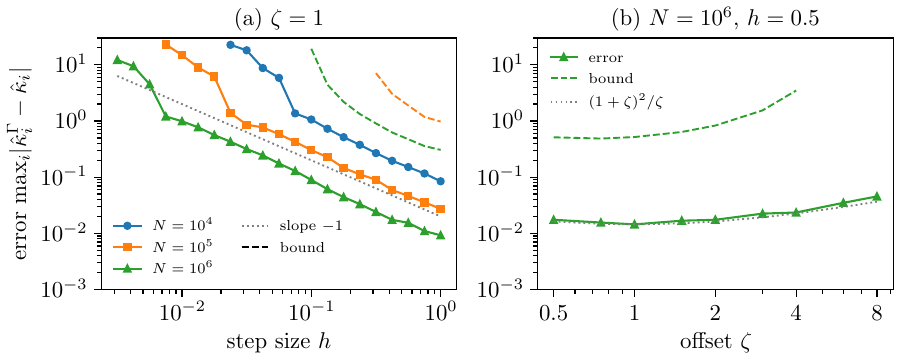}
  \caption{\Cref{alg:curv:recon_direction} on point clouds of the unit sphere. Error between the principal curvatures computed on $\Gamma$ and with the exact projection, for the same base point and rays, median over $200$ base points and frames; dashed: the bound of \cref{thm:curv:curv_ptcloud}. (a) Against the step size $h$, for the offset $\zeta = 1$. (b) Against $\zeta$, for $N = 10^6$ and $h = 0.5$; the dotted line is proportional to $(1+\zeta)^2/\zeta$.}
  \label{fig:curv:ptcloud_bound}
\end{figure}

\subsection{Practical adjustments for point clouds}\label{sec:curv:point_cloud_adjustment}
So on point clouds, the step size $h$ has to be long enough compared with the sample spacing, while \cref{thm:curv:error_alg2} prefers short steps, and in $\bbR^3$ the concentration of \cref{sec:curv:concentration} does not keep the slopes of the rays small. In our implementation, we therefore adjust \cref{alg:curv:recon_direction} as follows. The queries are restricted to the $N_{\mathrm{loc}} = 16^{d-1}$ nearest samples of $\tilde{\v y}_0 = \cP_\Gamma\v x_0$, and $\v x_0$ is drawn again if $\zeta<0.5$, since the factor $(1-\kappa_i\zeta_i)^2/\zeta$ in \cref{thm:curv:curv_ptcloud} grows as $\zeta\to0$. For the $j$-th ray of a frame $i$, the first crossing $h^1_{ij}$ is the smallest step for which both $\v x_0\pm h\v u_j$ have left the Voronoi cell of $\tilde{\v y}_0$. We compute it exactly with a few nearest neighbor queries on the two segments from $\v x_0$ along $\pm\v u_j$: if the samples closest to the two ends of a segment differ, the segment crosses their bisecting hyperplane at a point given by one linear equation, and the query at this bisecting point on the segment either confirms the crossing or returns a closer sample, with which we repeat the process on the shorter segment between $\v x_0$ and this point. Among $12$ random orthonormal frames, we take the one with the smallest $\max_jh^1_{ij}$, so that all its rays leave the cell quickly, which avoids rays with large slopes. The step $h$ is then chosen to be
\begin{equation}\label{eqn:curv:step_size}
  h = \frac{\median(\cR)}{\median(\cR_4)}\,\min_i\max_j h^1_{ij},
\end{equation}
where $\cR$ holds the distances from $\tilde{\v y}_0$ to its $N_{\mathrm{loc}}$ nearest samples and $\cR_4$ the $4^{d-1}$ smallest of them. For evenly spread samples the ratio is about $4$, so the projected rays cross a few Voronoi cells. The displacements are central differences, $\tilde{\v w}_i = (\bar{\v y}_i^+-\bar{\v y}_i^-)/2$, where $\bar{\v y}_i^\pm$ is the mean of the $2^{d-1}$ nearest samples of $\v x_0\pm h\v u_i$, and the offset is $\zeta = \norm{\v x_0-\tilde{\v y}_0}$, so that the base point stays on $\cM$. Optionally, the displacements are averaged over the five step sizes $0.5h, 0.75h, h, 1.25h, 1.5h$. We use~\cite{douze2024faiss} for the nearest neighbor queries.

\subsection{Computational cost and comparison with quadratic fitting}\label{sec:curv:pca_compare}

\Cref{alg:curv:recon_direction} requires $\cO(d)$ ray samples, $\cO(d)$ closest point projections within a local neighborhood, and an eigendecomposition of a $d\times d$ matrix. For a point cloud of size $N$, finding a local neighborhood of size $N_{\mathrm{loc}}$ costs $\cO(Nd)$ (\cref{sec:curv:point_cloud_adjustment}), and the projections within it cost $\cO(N_{\mathrm{loc}}d^2)$, so the total cost per query point is
\begin{equation}\label{eqn:curv:complexity_rays}
\cO(Nd) + \cO(N_{\mathrm{loc}}d^2) + \cO(d^3).
\end{equation}
An additional structural advantage is that the closest point projections in~\cref{alg:curv:recon_direction} can be computed independently over separate data partitions and aggregated afterward, making the method well suited to parallel and distributed implementations.

A classical and widely adopted approach for estimating extrinsic curvatures from point clouds is to locally approximate the surface by a quadratic polynomial and compute the curvatures of the fitted model. For each query point, it finds the $N_{\mathrm{loc}}$ nearest samples, estimates the tangent and normal spaces by principal component analysis (PCA) of them, with the last principal component as the normal of a hypersurface, expresses the samples in these local coordinates, fits a quadratic by least squares, and computes the principal curvatures of the fit.
Although quadratic fitting is intuitive and widely used in many low-dimensional settings, its accuracy critically depends on the quality of the estimated tangent and normal spaces. In turn, this estimation from PCA is highly sensitive to the choice of the neighborhood size $N_{\mathrm{loc}}$. Moreover, the computational cost of the method grows rapidly with the ambient dimension: finding a local neighborhood of size $N_{\mathrm{loc}}$ costs $\cO(Nd)$, or $\cO(N\log N)$ for small $d$; PCA on the neighborhood $\cO(N_{\mathrm{loc}}d^2 + d^3)$; projection onto the tangent space $\cO(N_{\mathrm{loc}}d^2)$; and quadratic fitting, which requires solving a linear system of size $\cO(d^2)$, $\cO(N_{\mathrm{loc}}d^4 + d^6)$. As a result, the overall complexity is
\begin{equation}\label{eqn:curv:complexity_pca}
  \cO(Nd) + \cO(N_{\mathrm{loc}}d^4) + \cO(d^6),
\end{equation}
compared with $\cO(Nd) + \cO(N_{\mathrm{loc}}d^2) + \cO(d^3)$ for~\cref{alg:curv:recon_direction}. For the same neighborhood size $N_{\mathrm{loc}}$, the last two terms are smaller by the factors $d^2$ and $d^3$. Quadratic fitting also relies on explicit construction of local neighborhoods, which may be infeasible when the point cloud is distributed across multiple machines.

Constructing sufficiently dense point clouds in high dimensions is computationally prohibitive on a standard laptop, so we restrict the numerical comparison to point clouds sampled from the unit sphere and an ellipsoid of semi-axes $(3,\! 2,\! 1)$ in $\bbR^3$, with the goal of illustrating qualitative differences between the methods. The results are shown in~\cref{fig:curv:pca_compare}.
\begin{figure}[htbp]
  \centering
  \includegraphics[width=0.9\textwidth]{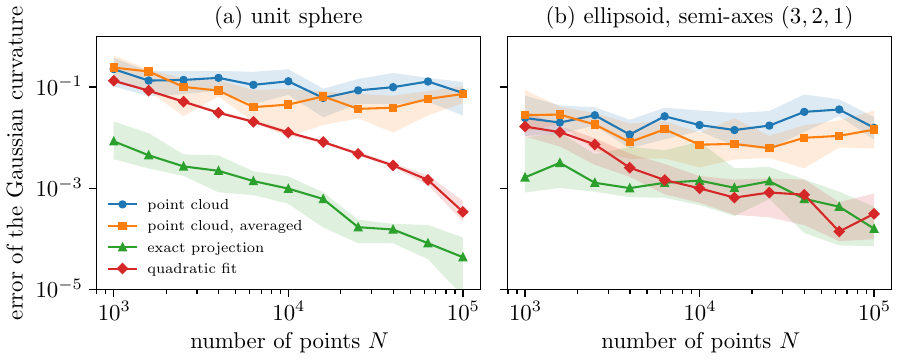}
  \caption{Comparison on sphere and ellipsoid point clouds with varying densities. Absolute errors of Gaussian curvature over $16$ random trials, for $N$ uniform samples of (a) the unit sphere and (b) the ellipsoid with semi-axes $(3,2,1)$; the lines are the medians and the shaded regions are the quartiles. Point cloud and point cloud, averaged: \cref{alg:curv:recon_direction} with the adjustments of \cref{sec:curv:point_cloud_adjustment}, with one step size and averaged over five, respectively. Exact projection: the same base points and rays with the exact closest point projection and $h = 10^{-2}$; its error comes from the offset $\zeta = \norm{\v x_0-\tilde{\v y}_0}$, which is measured to the sample $\tilde{\v y}_0$ nearest to $\v x_0$ from the point cloud instead of to the surface. Quadratic fit: after principal component analysis of the $64$ nearest samples.}
  \label{fig:curv:pca_compare}
\end{figure}

From the comparison, we observe that PCA-based quadratic fitting performs well on these relatively dense and low-dimensional datasets, achieving both small bias and low variance. The ray-based methods incur additional error due to the inexactness of closest point projections on point clouds. Nevertheless, the resulting error remains bounded within a modest range, which shows that the algorithm remains robust in this discrete setting. This error decreases little with $N$, as \cref{thm:curv:curv_ptcloud} predicts: the step $h$ follows the sample spacing, so $\varepsilon/h$ stays fixed. For the error to decrease with $N$, $h$ has to shrink more slowly than the sample spacing. With $h\sim\sqrt\varepsilon$, for example, the error $\cO(\varepsilon/h)$ from the point cloud and the error $\cO(h)$ of \cref{thm:curv:error_alg2} are both $\cO(\sqrt\varepsilon)$.

For the submanifolds of \cref{sec:curv:codim}, the projections onto a point cloud are again nearest neighbor queries, and the bounds of \cref{sec:curv:error_point_cloud} carry over to any codimension, since the proofs of these bounds use only the quadratic growth of \cref{thm:curv:quad_growth} and a bound on the second derivatives of the local graph. An $\varepsilon$-sample of $\cM$ then needs of the order of $\varepsilon^{-d_{\mathrm{in}}}$ points, a number that depends on the dimension of $\cM$ and not on $d$. In our first tests, however, the errors on point clouds grew with the codimension, since the discretization error of the nearest neighbor displacements enters each of the $d_{\mathrm{co}}$ normal directions. We leave the point-cloud version of \cref{alg:curv:codim} to future work.

Finally, we emphasize that our framework is not designed as a direct replacement for PCA-based methods on point clouds, but rather as a complementary approach that applies whenever a reasonably well-defined closest point projection is available. This includes not only point clouds, but also implicit surfaces, level-set representations, and hybrid geometric data. In this sense, the projection-based method provides a flexible approach for curvature computation across a wide range of geometric representations.

\section{Summary}\label{sec:curv:conclusion}

The main point of this paper is that the closest point projections of a few short rays give the extrinsic curvature of a manifold in any dimension and codimension, without local coordinate charts or local fitting. The argument is the following.
\begin{enumerate}
  \item A ray from a point $\v x_0$ near $\cM$ induces a curve on $\cM$ through the closest point projection. Its velocity at $\v y_0 = \cP\v x_0$ is the tangential part of the ray scaled exactly by $1/(1-\zeta\kappa_i)$ along the $i$-th principal direction at every offset $\zeta$ within the reach (\cref{thm:curv:parametrization,thm:curv:dP_codim}), where $\kappa_i$ are the principal curvatures in the direction of the offset. The normal part of its acceleration is the normal curvature of $\cM$ in the direction of the curve (\cref{rmk:curv:meusnier}).
  \item The induced curve is not a geodesic in general. For the osculating paraboloid, it is a geodesic at $\v y_0$ for tangent rays and for rays whose tangential parts are principal directions, and when $\cM$ is locally convex near $\v y_0$ and $\v x_0$ lies along the outward normal (\cref{asm:curv:convex_side}), it converges to a geodesic as $\zeta\to\infty$ (\cref{thm:curv:local_geodesic,thm:curv:r_convergence}). The third-order terms of $\cM$ add a geodesic curvature that does not vanish in this limit (\cref{thm:curv:third_order}); it belongs to the induced curve at $\v y_0$, so a smaller step size does not remove it.
  \item The second fundamental form can be reconstructed by least squares from $\cO(d^2\ln d)$ induced curves (\cref{alg:curv:recon_2nd}). With the curvatures of the induced curves, the error is $\cO(\ln d/\zeta^2)$ for the osculating paraboloid (\cref{thm:curv:error_alg1}), and of order $1/(d\zeta^2)$ in practice (\cref{rmk:curv:alg1_rate}); with their normal curvatures (\cref{rmk:curv:meusnier}), the bias disappears and only the error of the finite differences remains, below $2\times10^{-7}$ in \cref{fig:curv:high_dim}(c); no convexity is needed.
  \item The principal curvatures follow directly from the eigenvalues of a $d\times d$ matrix built from $d$ rays, at a cost of $\cO(d^3)$ besides the projections, without normal information and without any convexity assumption (\cref{thm:curv:recon_direction,thm:curv:error_alg2}). The relative error is $\cO(h^2/\zeta)$ for tangent rays and $\cO\big((h/\zeta)\sqrt{\ln(d)/d}\big)$ for random rays, with probability at least $1-1/d$.
  \item In codimension $d_{\mathrm{co}}$, a probing point only recovers the shape operator in the normal direction along its offset. Other points on the normal disk produce the other normal directions, one at a time, and the rays of the first probing point can estimate the tangent space (\cref{alg:curv:codim,thm:curv:codim}). With the Euclidean projection, the error is $\cO(h/\zeta)+\cO(h)$ at every offset within the reach, and the central second differences of \cref{rmk:curv:codim_meusnier} gave errors decreasing like $h^2$ in our tests.
\end{enumerate}
The experiments confirm these points. On the unit sphere, the error of the direct algorithm decreases like $h/\zeta$, as predicted, and like $h^2$ with central differences (\cref{fig:curv:sphere_surface_compare}). On the graph of \eqref{eqn:curv:test_function}, it recovers all principal curvatures with errors below $2.1\times10^{-4}$ up to $d = 128$ from $d+1$ projections, and the reconstruction of the second fundamental form shows the bias of the third-order terms, which the normal curvature removes (\cref{fig:curv:high_dim}). For surfaces in up to $100$ dimensions, the second fundamental form is recovered with relative errors below $1.4\times10^{-3}$ for $h = 10^{-3}$ in every codimension (\cref{fig:curv:codim}). On point clouds, the nearest neighbor projection adds an error of order $\varepsilon/h$ for an $\varepsilon$-sample, which the experiments on spheres confirm (\cref{thm:curv:curv_ptcloud,fig:curv:ptcloud_bound}); in $\bbR^3$, the direct algorithm with the practical adjustments stays within a modest error, but quadratic fitting after local principal component analysis is more accurate there (\cref{fig:curv:pca_compare}).

Overall, the closest point projection connects the differential geometry of a manifold with practical computation: the curvature follows from closest point queries alone, with error bounds for exact projections and for point clouds. The curvature of the data manifold changes even local linear regression~\cite{liu2023linear}, and the estimates of this paper are a first step to measure it from data.

Some questions are left for future work: an analysis of the estimates on point clouds that informs the choice of step sizes $h$ and offset $\zeta$, the point-cloud version of \cref{alg:curv:codim}, and adaptive choices of rays and step sizes on surfaces with strongly varying curvature.

\appendix
\section{Error analysis of the reconstruction of the second fundamental form}\label{sec:curv:app_alg1}
This appendix gives the conditioning of the design matrix $\Psi$ of \cref{sec:curv:alg_recon_error} and the proof of \cref{thm:curv:error_alg1}. Each row of $\Psi$ is the image $\nu_2(\v t_i) = \v t_i\otimes\v t_i$ under the Veronese mapping of degree $2$, which has the following instrumental property in analyzing $\sigma_{\min}(\Psi)$.
\begin{lemma}[Isotropic embedding of the unit sphere]\label{thm:curv:isotropic}
The degree-$2$ Veronese embedding $\nu_2(\v r)\in\bbR^{k^2}$ of uniformly distributed $\v r\in\bbS^{k-1}$ is isotropic, i.e., when restricted to its intrinsic embedding subspace $V$ (the vectorized symmetric traceless $k\times k$ matrices),
\begin{equation}\label{eqn:curv:isotropic}
  \Cov[\nu_2(\v r)]\evalmid{}{V} = \frac{2}{k(k+2)} I_V,
\end{equation}
where $\Cov$ denotes the covariance operator and $I_V$ is the identity on $V$.
\end{lemma}
\begin{proof}
Since $\nu_2(\v r)=\v r\otimes \v r = \vect(\v r\v r^T)$, where $\vect$ is the vectorization operator and is a linear isomorphism, $\bbE[\nu_2(\v r)] = \vect\big(\bbE[\v r\v r^T]\big)$. By rotational invariance of the uniform distribution on $\bbS^{k-1}$, $\bbE[\v r\v r^T] = \frac{1}{k}I_k$, so $\bbE[\nu_2(\v r)] = \vect(\frac{1}{k}I_k)$. Define the centered variable
\[
  T(\v r) \defeq \v r\v r^T - \frac{1}{k}I_k,
  \quad \bs\tau(\v r) \defeq \vect\big(T(\v r)\big).
\]
Since $\tr(\v r\v r^T)=1$, we have $\tr(T)=0$, so $T$ lies in the space of real symmetric traceless matrices, a $(\frac{k(k+1)}{2}-1)$-dimensional subspace of $\bbR^{k\times k}$, and $\bs\tau$ lives in a subspace $V\subset\bbR^{k^2}$ of the same dimension.
For any $U\in O(k)$, the properties of the Kronecker product~\cite{schacke2004kronecker, magnus2019matrix} give
$\nu_2(U\v r) = (U\otimes U)\nu_2(\v r)$, where $U\otimes U$ is orthogonal since $(U\otimes U)^{-1} = U^T\otimes U^T = (U\otimes U)^T$. Likewise $\bs\tau(U\v r) = (U\otimes U)\bs\tau(\v r)$. Since $U\v r\disteq\v r$, we conclude $\bs\tau(\v r)\disteq(U\otimes U)\bs\tau(\v r)$: the covariance operator $\Cov[\bs\tau]$ is invariant under the natural orthogonal action $U\otimes U$ of $O(k)$ on $V$. From representation theory~\cite{fulton2013representation}, the space $V$ of symmetric traceless matrices forms an irreducible representation of $O(k)$. By Schur's Lemma~\cite{dummit2004abstract}, $\Cov[\bs\tau]\evalmid{}{V} = c I_V$ for some constant $c>0$. To determine $c$, note $\normbig{\v r\v r^T-\frac{1}{k}I_k}_F^2 = 1 - \frac{2}{k} + \frac{1}{k} = 1-\frac1k$, so
\[
  \tr\big(\Cov[\bs\tau]\big) = \frac{k-1}{k} = c\Big(\frac{k(k+1)}{2}-1\Big)
  \implies c = \frac{2}{k(k+2)}.
\]
\end{proof}

The isotropic property indicates that the sample covariance concentrates around its mean~\cite{rudelson1999random,oliveira2010sums,tropp2012user}. Because the covariance is the identity scaled by $2/(k(k+2))$, this concentration controls $\sigma_{\min}$ at the scale $\sqrt{2M/(k(k+2))}$.

\begin{proposition}[Conditioning of the design]\label{thm:curv:conditioning}
Let $\v u_1,\dots,\v u_M$ be independent and uniform on $\bbS^{k-1}$, $k\ge2$, and let $\Psi_0\in\bbR^{M\times k^2}$ have rows $\nu_2(\v u_i)^T$. There is an absolute constant $C$ such that, if $M\ge Ck^2\ln(k/p)$ for some $p\in(0,1)$, then with probability at least $1-p$
\begin{equation}\label{eqn:curv:conditioning}
  \norm{\Psi_0\bs\xi}\ge\sqrt{\frac{M}{2k(k+2)}}\,\norm{\bs\xi}
  \qquad\text{for every vectorized symmetric $k\times k$ matrix $\bs\xi$},
\end{equation}
that is, $\norm{\Psi_0^\dagger}\le\sqrt{2k(k+2)/M}$ on the symmetric matrices.
\end{proposition}
\begin{proof}
Write $\nu_2(\v u_i) = \vect(\frac1kI_k)+\bs\tau_i$ with $\bs\tau_i\in V$ as in \cref{thm:curv:isotropic}, $\norm{\bs\tau_i}^2 = 1-1/k$, and let $\v e = \vect(I_k)/\sqrt k$, a unit vector orthogonal to $V$; the symmetric matrices are $V\oplus\operatorname{span}\{\v e\}$. For $\bs\xi = \bs\xi_V+a\v e$ we have $\inpd{\nu_2(\v u_i)}{\bs\xi} = \inpd{\bs\tau_i}{\bs\xi_V}+a/\sqrt k$, so
\[
  \norm{\Psi_0\bs\xi}^2 = \sum_i\inpd{\bs\tau_i}{\bs\xi_V}^2+\frac{2a}{\sqrt k}\Big\langle\sum_i\bs\tau_i,\bs\xi_V\Big\rangle+\frac{Ma^2}{k}.
\]
The matrices $\bs\tau_i\bs\tau_i^T$ on $V$ are independent, positive semidefinite, of norm at most $1$, with mean $cI_V$, $c = 2/(k(k+2))$. The matrix Chernoff inequality~\cite[Thm.~1.1]{tropp2012user} gives $\sum_i\inpd{\bs\tau_i}{\bs\xi_V}^2\ge\frac12Mc\norm{\bs\xi_V}^2$ for all $\bs\xi_V$, with probability at least $1-\dim V\,e^{-0.15Mc}$. The vectors $\bs\tau_i$ are independent, centered and bounded by $1$, so by the bounded differences inequality $\norm{\sum_i\bs\tau_i}\le\beta\sqrt M$ with $\beta = 1+\sqrt{2\ln(2/p)}$, with probability at least $1-p/2$. On both events, $2xy\le\frac12x^2+2y^2$ with $x = \abs a\sqrt{M/k}$, $y = \beta\norm{\bs\xi_V}$ gives
\[
  \norm{\Psi_0\bs\xi}^2\ge\Big(\frac{Mc}2-2\beta^2\Big)\norm{\bs\xi_V}^2+\frac{Ma^2}{2k}\ge\frac{Mc}{4}\big(\norm{\bs\xi_V}^2+a^2\big),
\]
as soon as $Mc\ge8\beta^2$; note $M/(2k)\ge Mc/4$. Both conditions, and $\dim V\,e^{-0.15Mc}\le p/2$, hold when $M\ge Ck^2\ln(k/p)$.
\end{proof}

\Cref{fig:curv:Psi_rate} confirms the scale of $\norm{\Psi^\dagger}$ for uniformly distributed unit directions: for $M = 2d^2$ it tends to a constant close to $1$, and for fixed $d$ it decreases like $(\sqrt{2c}-1)^{-1}$ in the multiplier $c$ of $M = cd^2$, in agreement with the edge of the Marchenko--Pastur law for rows with covariance $2/(d(d+2))$.
\begin{figure}[htbp]
  \centering
  \includegraphics[width=0.9\textwidth]{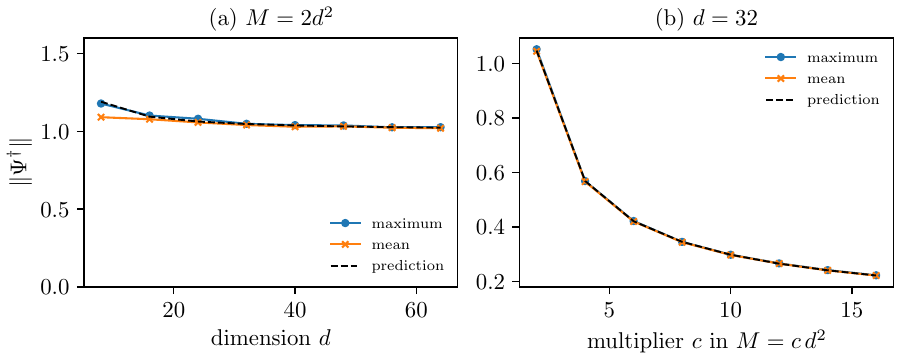}
  \caption{Norm of $\Psi^\dagger$ on the symmetric matrices for rows $\v r_i\otimes\v r_i$, $\v r_i$ uniform on $\bbS^{d-1}$: maximum and mean over $5$ simulations. (a) $M = 2d^2$, varying $d$. (b) $d = 32$, varying the multiplier $c$ in $M = cd^2$. Dashed: the prediction $\sqrt{d(d+2)/2}/(\sqrt M-\sqrt{d(d+1)/2})$.}
  \label{fig:curv:Psi_rate}
\end{figure}

\begin{proof}[Proof of \cref{thm:curv:error_alg1}]
Work in an orthonormal basis of $\cT_{\v y}\cM$, so that the unknown is a symmetric $(d-1)\times(d-1)$ matrix. By rotational symmetry about $\v n$, the normalized tangential parts $\v u_i$ of the rays are independent and uniform on the unit sphere of $\cT_{\v y}\cM$, and by \cref{thm:curv:parametrization} the measured directions are $\v t_i = G\v u_i/\norm{G\v u_i}$ with $G = (I-2\zeta A)^{-1}$. Hence $\v t_i\otimes\v t_i = (G\otimes G)(\v u_i\otimes\v u_i)/\norm{G\v u_i}^2$, and since $G$ is symmetric, the $i$-th entry of $\Psi\bs\xi$ is that of $\Psi_0(G\otimes G)\bs\xi$ divided by $\norm{G\v u_i}^2\le\sigma_{\max}(G)^2$. The map $G\otimes G$ sends symmetric matrices to symmetric matrices and shrinks them by at most the factor $\sigma_{\min}(G)^2$, so for every symmetric $\bs\xi$
\[
  \norm{\Psi\bs\xi}\ge\frac{\norm{\Psi_0(G\otimes G)\bs\xi}}{\sigma_{\max}(G)^2}
  \ge\sqrt{\frac{M}{2k(k+2)}}\,\frac{\sigma_{\min}(G)^2}{\sigma_{\max}(G)^2}\,\norm{\bs\xi},
  \qquad k = d-1,
\]
on the event of \cref{thm:curv:conditioning}. Under \cref{asm:curv:convex_side} the condition number of $G$ is $(1+2\zeta\abs{\lambda_{\max}})/(1+2\zeta\abs{\lambda_{\min}})\le\tau_A$. With $p = 1/d$ in \cref{thm:curv:conditioning}, $\norm{\Psi^\dagger}\le\tau_A^2\sqrt{2(d-1)(d+1)/M}$ with probability at least $1-1/d$. Combining with $\norm{\Omega}\le\sqrt M\max_i\omega_i$, the factor $\sqrt M$ cancels:
\begin{equation}\label{eqn:curv:alg1_bound}
  \norm{\hat S_{\v y}-S_{\v y}}_F = \norm{\xi-\hat\xi}\le\norm{\Psi^\dagger}\norm{\Omega}\le\tau_A^2\sqrt{2(d^2-1)}\,\max_i\omega_i .
\end{equation}
By the proof of \cref{thm:curv:r_convergence}, $\omega_i\le\mathrm{Res}_i^2/(2\abs{\kappa_{n,i}})\lesssim\gamma_i^2\tau_A^2(\tau_A-1)^2/(2\zeta^2\abs{\kappa_{n,i}})$, and with probability at least $1-1/d$ all $\gamma_i^2\le4\ln(2dM)/(d-2)$ by \eqref{eqn:curv:slope_bound}. Since $\ln(2dM) = \cO(\ln d)$, this gives $\norm{\hat S_{\v y}-S_{\v y}}_F = \cO(\ln(d)/\zeta^2)$ with probability at least $1-2/d$. Finally, since both matrices are symmetric, the Hoffman--Wielandt inequality~\cite{hoffman1953variation} gives $\sum_i\abs{\hat\kappa_i-\kappa_i}^2\le\norm{\hat S_{\v y}-S_{\v y}}_F^2$, and the individual bound follows.
\end{proof}

\section*{Acknowledgments}
Liu and Tsai are partially supported by National Science Foundation Grant DMS-2110895. Tsai is also partially supported by Army Research Office Grant W911NF2320240. We thank Dan Knopf and Ziheng Chen for their insights and discussions in the early stage of this work. We are tremendously saddened by the untimely passing of Dan Knopf, and we extend our deepest condolences to his family.

\bibliographystyle{siamplain}
\bibliography{references}

\end{document}